\documentclass[reqno]{amsart}

\usepackage[ansinew]{inputenc}
\usepackage{amssymb}
\usepackage{amsmath}
\usepackage{graphicx}

\makeatletter \@addtoreset{equation}{section} \makeatother

\renewcommand\thefigure{\thesection.\@arabic\c@figure}
\renewcommand\thetable{\thesection.\@arabic\c@table}

\newtheorem{theorem}{Theorem}[section]
\newtheorem{lemma}[theorem]{Lemma}
\newtheorem{proposition}[theorem]{Proposition}
\newtheorem{corollary}[theorem]{Corollary}

\newtheorem{remark}[theorem]{Remark}

\newcommand{\bb}[1]{{\mathbb #1}}
\newcommand{\mc}[1]{{\mathcal #1}}

\newcommand{\mb}[1]{{\mathbf #1}}
\newcommand{\bs}[1]{{\boldsymbol #1}}

\newcommand{\<}{\langle}
\renewcommand{\>}{\rangle}

\newcommand{\ca}{\mathcal{A}}

\newcommand{\cc}{\mathcal{C}}

\newcommand{\cl}{\mathcal{L}}
\newcommand{\cm}{\mathcal{M}}

\newcommand{\C}{\mathbb{C}}

\newcommand{\R}{\mathbb{R}}

\let\n=\eta
\let\r=\rho

\let\G=\Gamma

\newcommand{\1}{\,\rlap{\small 1}\kern.13em 1}

\newcommand{\sqr}[2]{{\vcenter{\hrule height.#2pt%
                      \hbox{\vrule width.#2pt height#1pt\kern#1pt%
                            \vrule width.#2pt}%
                      \hrule height.#2pt}}}
\newcommand{\cqfd}{\hfill$\mathchoice\sqr46\sqr46\sqr{1.5}2\sqr{1}2$\par}

\renewcommand{\limsup}{\mathop{\overline{\hbox{\rm lim}}}}
\renewcommand{\liminf}{\mathop{\underline{\hbox{\rm lim}}}}

\newcommand{\Es}{\mathbb{E}}
\newcommand{\Pb}{\mathbb{P}}
\newcommand{\mss}{\mu_{\text{ss}}^N}

\newcommand{\Qss}{{\text{\bf Q}}_{\text{ss}}}

\newcommand{\bro}{{\bar{\rho}}}

\title[Boundary driven exclusion processes]{Hydrostatics and dynamical
  large deviations of boundary driven gradient symmetric exclusion
  processes}

\author{J. Farfan, C. Landim \and M. Mourragui}

\address{{\rm J. Farfan} \newline
IMPA, Estrada Dona Castorina 110,
CEP 22460 Rio de Janeiro, Brasil
\newline e-mail: \rm \texttt{jonathan@impa.br}}

\address{{\rm C. Landim} \newline
IMPA, Estrada Dona Castorina 110,
CEP 22460 Rio de Janeiro, Brasil
\newline e-mail: \rm \texttt{landim@impa.br}, \rm \texttt{Claudio.Landim@univ-rouen.fr}
}

\address{{\rm M. Mourragui}\newline
LMRS, UMR 6085 CNRS-Universit\'e de Rouen
Avenue de l'Universit\'e, BP.12.  Technop\^ole du Madrillet
F76801 Saint-'Etienne-du-Rouvray. France
\newline e-mail: \rm \texttt{Mustapha.Mourragui@univ-rouen.fr}}

\begin{document}

\noindent  


\begin{abstract}
  We prove hydrostatics of boundary driven gradient exclusion
  processes, Fick's law and we present a simple proof of the dynamical
  large deviations principle which holds in any dimension.
\end{abstract}

\maketitle


\section{introduction}


Statical and dynamical large deviations principles of boundary driven
interacting particles systems has attracted attention recently as a
first step in the understanding of nonequilibrium thermodynamics (cf.
\cite{BDGJL7, bd, Der} and references therein).

This article has two purposes. First, inspired by the dynamical
approach to stationary large deviations, introduced by Bertini et al.
in the context of boundary driven interacting particles systems
\cite{bdgjl2}, we present a proof of the hydrostatics based on the
hydrodynamic behaviour of the system and on the fact that the
stationary profile is a global attractor of the hydrodynamic
equation. 

More precisely, if $\bar \rho$ represents the stationary density
profile and $\pi^N$ the empirical measure, to prove that $\pi^N$
converges to $\bar\rho$ under the stationary state $\mu^N_{ss}$, we
first prove the hydrodynamic limit stated as follows. If we start from
an initial configuration which has a density profile $\gamma$, on the
diffusive scale the empirical measure converges to an absolutely
continuous measure, $\pi(t,du) = \rho(t,u) du$, whose density $\rho$
is the solution of the parabolic equation
\begin{equation*}
\left\{ 
\begin{array}{l}
\partial_t \rho = (1/2) \nabla \cdot D(\rho) \nabla \rho\, , \\ 
\rho (0 ,\cdot) =\; \gamma (\cdot) \, ,   \\
\rho(t,\cdot) = b(\cdot) \text{ on $\Gamma$}\;,
\end{array}
\right. 
\end{equation*}
where $D$ is the diffusivity of the system, $\nabla$ the gradient, $b$
is the boundary condition imposed by the stochastic dynamics and
$\Gamma$ is the boundary of the domain in which the particles evolve.
Since for all initial profile $0\le \gamma \le 1$, the solution
$\rho_t$ is bounded above, resp. below, by the solution with initial
condition equal to $1$, resp. $0$, and since these solutions converge,
as $t\uparrow\infty$, to the stationary profile $\bar\rho$,
hydrostatics follows from the hydrodynamics and the weak compactness
of the space of measures.

The second contribution of this article is a simplification of the
proof of the dynamical large deviations of the empirical measure.  The
original proof \cite{kov, dv, KL} relies on the convexity of the rate
functional, a very special property only fulfilled by very few
interacting particle systems as the symmetric simple exclusion
process. The extension to general processes \cite{q, QRV, blm1} is
relatively technical. The main difficulty appears in the proof of the
lower bound where one needs to show that any trajectory $\lambda_t$,
$0\le t\le T$, with finite rate function, $I_T(\lambda)<\infty$, can
be approximated by a sequence of smooth trajectories $\{\lambda^n :
n\ge 1\}$ such that
\begin{equation}
\label{int2}
\lambda^n \longrightarrow \lambda \quad\text{and}\quad 
I_T(\lambda^n)  \longrightarrow  I_T(\lambda)\;.
\end{equation}

This property is proved by approximating in several steps a general
trajectory $\lambda$ by a sequence of profiles, smoother at each step,
the main ingredient being the regularizing effect of the hydrodynamic
equation. This part of the proof is quite elaborate and relies on
properties of the Green kernel associated to the second order
differential operator.

We propose here a simpler proof. It is well known that a path
$\lambda$ with finite rate function may be obtained from the
hydrodynamical path through an external field. More precisely, if
$I_T(\lambda)<\infty$, there exists $H$ such that
\begin{equation*}
I_T(\lambda) \;=\; \frac 12 \int_0^T dt \int \sigma(\lambda_t) \, 
[\nabla H_t]^2 \, dx\;,
\end{equation*}
 where $\sigma$ is the mobility of the system and $H$ is related to
$\lambda$ by the equation
\begin{equation}
\label{int1}
\left\{
\begin{array}{l}
{\displaystyle
\partial_t \lambda - (1/2) \nabla \cdot D(\lambda)\nabla \lambda 
\;=\; - \nabla \cdot [\sigma(\lambda) \nabla H_t] }\\
{\displaystyle H(t,\cdot)=0 \text{ at the boundary}\;.}
\end{array}
\right.  
\end{equation}
This is an elliptic equation for the unknown function $H$ for each
$t\ge 0$.  Note that the left hand side of the first equation is the
hydrodynamical equation. Instead of approximating $\lambda$ by a
sequence of smooth trajectories, we show that approximating $H$ by a
sequence of smooth functions, the corresponding smooth solutions of
\eqref{int1} converge in the sense \eqref{int2} to $\lambda$. This
approach, closer to the original one, simplifies considerably the
proof of the hydrodynamical large deviations.


\section{Notation and Results}

\label{sec2}

Fix a positive integer $d\ge 2$.  Denote by $\Omega$ the open set $(-1,1)
\times \bb T^{d-1}$, where $\bb T^{k}$ is the $k$-dimensional torus
$[0,1)^k$, and by $\Gamma$ the boundary of $\Omega$: $\Gamma = \{(u_1,
\dots , u_d)\in [-1,1] \times \bb T^{d-1} : u_1 = \pm 1\}$.

For an open subset $\Lambda$ of $\bb R\times\bb T^{d-1}$, $\mc C^m
(\Lambda)$, $1\leq m\leq +\infty$, stands for the space of
$m$-continuously differentiable real functions defined on $\Lambda$.
Let $\mc C^m_0 (\Lambda)$ (resp. $\mc C^m_c (\Lambda)$), $1\leq m\leq
+\infty$, be the subset of functions in $\mc C^m (\Lambda)$ which
vanish at the boundary of $\Lambda$ (resp. with compact support in
$\Lambda$).

Fix a positive function $b: \Gamma \to \bb R_+$. Assume that there
exists a neighbourhood $V$ of $\Omega$ and a smooth function $\beta : V
\to (0,1)$ in $\mc C^2(V)$ such that $\beta$ is bounded below by a
strictly positive constant, bounded above by a constant smaller than
$1$ and such that the restriction of $\beta$ to $\Gamma$ is equal to
$b$. 

For an integer $N\ge 1$, denote by $\bb T_N^{d-1}=\{0,\dots,
N-1\}^{d-1}$, the discrete $(d-1)$-dimensional torus of length $N$.
Let $\Omega_N=\{-N+1,\ldots,N-1\} \times \bb T_N^{d-1}$ be the
cylinder in $\bb Z^d$ of length $2N-1$ and basis $\bb T_N^{d-1}$ and let
$\G_N=\{(x_1, \dots, x_{d}) \in \bb Z\times \bb T_N^{d-1}\,|\, x_1 =\pm
(N-1)\}$ be the boundary of $\Omega_N$.  The elements of $\Omega_N$
are denoted by letters $x,y$ and the elements of $\Omega$ by the
letters $u, v$.

We consider boundary driven symmetric exclusion processes on
$\Omega_N$.  A configuration is described as an element $\n$ in
$X_N=\{0,1\}^{\Omega_N}$, where $\n (x)=1$ (resp. $\n (x)=0$) if site
$x$ is occupied (resp. vacant) for the configuration $\eta$.  At the
boundary, particles are created and removed in order for the local
density to agree with the given density profile $b$.

The infinitesimal generator of this Markov process can be decomposed
in two pieces:
\begin{equation}
\label{eq:1}
\cl_N=\cl_{N,0}+\cl_{N,b}\;,
\end{equation}
where $\cl_{N,0}$ corresponds to the bulk dynamics and $\cl_{N,b}$ to
the boundary dynamics. The action of the generator $\cl_{N,0}$ on
functions $f: X_N\to \R$ is given by
\begin{equation*}
\big(\cl_{N,0}f\big) (\n)=\sum_{i=1}^d \sum_{x}r_{x,x+e_i} (\n) 
\left[f(\n^{x,x+e_i})-f(\n)\right] , 
\end{equation*} 
where $(e_1,\ldots,e_d)$ stands for the canonical basis of $\R^d$ and
where the second sum is performed over all $x\in \bb Z^d$ such that
$x,x+e_i\in\Omega_N$. For $x,y\in \Omega_N$, $\n^{x,y}$ is the
configuration obtained from $\n$ by exchanging the occupations variables
$\n(x)$ and $\n(y)$:
\begin{equation*}
\n^{x,y}(z)= 
\left\{ 
\begin{array}{lll}
\n (y) & \text{if}\ z= x \,, \\
\n(x) & \text{if}\ z=y \,, \\
\n(z) & \text{if}\ z\not=x,y \, .
\end{array}
     \right. 
\end{equation*}

For $a>- 1/2$, the rate functions $r_{x,x+e_i} (\n)$ are given by
\begin{equation*}
r_{x,x+e_i} (\n) \;=\; 1 \;+\; a \big\{ \n(x-e_i) + \n(x+2e_i)\big\}
\end{equation*}
if $x-e_i$, $x+2e_i$ belongs to $\Omega_N$.  At the boundary, the
rates are defined as follows. Let $\check{x}=(x_2,\cdots ,x_d) \in
\bb T_N^{d-1}$. Then,
\begin{equation*}
\begin{array}{l} 
\vphantom{\Big\{}
r_{(-N+1,\check{x}),(-N+2,\check{x})} (\n) \;=\;
1 \;+\; a \big\{ \n(-N+3,\check{x}) + b (-1, \check{x}/N) \big\} 
\, ,  \\
\vphantom{\Big\{}
r_{(N-2,\check{x}),(N-1,\check{x}) }(\eta)\;=\;
1 \;+\; a \big\{ \n(N-3,\check{x}) +  b(1, \check{x}/N)\big\} \, .
\end{array}
\end{equation*}

The non-conservative boundary dynamics can be described as
follows. For any function $f: X_N \to \bb R$,
\begin{equation*}
\left(\cl_{N,b} f\right)(\n) = \sum_{x\in\G_N} C^b( x ,\eta) 
\big[f(\eta^x) -f(\eta)\big]\, ,
\end{equation*}
where $\n^{x}$ is the configuration obtained from $\n$ by flipping the
occupation variable at site $x$:
\begin{equation*}
\n^x(z)= \left\{ \begin{array}{ll}
\n (z) & \text{if}\ z\ne x \\
1-\n(x) & \text{if}\ z=x\end{array}
\right. 
\end{equation*}
and the rates $C^b(x, \cdot)$ are chosen in order for the Bernoulli
measure with density $b(\cdot)$ to be reversible for the flipping
dynamics restricted to this site:

\begin{eqnarray*}
C^b \big((-N+1,\check{x}),\eta\big) & = &
\n (-N+1,\check{x}) \big[ 1-b(-1,\check{x}/N)\big]
\\
&& \qquad\qquad+\big[1-\n (-N+1,\check{x})\big] b(-1,\check{x}/N) \, ,
\\
C^b \big((N-1,\check{x}),\eta\big) & = &
\n (N-1,\check{x}) \big[ 1-b(1,\check{x}/N)\big]
\\
&& \qquad\qquad+\big[1-\n (N-1,\check{x})\big] b(1,\check{x}/N) \, ,
\end{eqnarray*}
where $\check{x}=(x_2,\cdots,x_d)\in\bb T^{d-1}_N$, as above.

Denote by $\{\eta_t : t\ge 0\}$ the Markov process associated to the
generator $\cl_N$ \emph{speeded up} by $N^2$. For a smooth function
$\rho :\Omega\to (0,1)$, let $\nu_{\rho(\cdot)}^N$ be the Bernoulli
product measure on $X_N$ with marginals given by
\begin{equation*}
\nu_{\r(\cdot)}^N(\n(x)=1) = \r(x/N)\;.
\end{equation*}
It is easy to see that the Bernoulli product measure associated to
any constant function is invariant for the process with generator
$\cl_{N,0}$.  Moreover, if $b(\cdot)\equiv b$ for some constant $b$
then the Bernoulli product measure associated to the constant density
$b$ is reversible for the full dynamics $\cl_N$.

\subsection{Hydrostatics}

Denote by $\mss$ the unique stationary state of the irreducible Markov
process $\{\eta_t : t\ge 0\}$. We examine in Section \ref{sec3} the
asymptotic behavior of the empirical measure under the stationary
state $\mss$.

Let $\cm = {\cm}(\Omega)$ be the space of positive measures on
$\Omega$ with total mass bounded by $2$ endowed with the weak topology.
For each configuration $\eta$, denote by $\pi^N = \pi^N(\eta)$ the
positive measure obtained by assigning mass $N^{-d}$ to each particle
of $\eta$~:
\begin{equation*} 
\pi^N \, =\, N^{-d}\sum_{x\in\Omega_N}\eta(x)\, \delta_{x/N}\; ,
\end{equation*}
where $\delta_{u}$ is the Dirac measure concentrated on $u$. 

To define rigorously the quasi-linear elliptic problem the empirical
measure is expected to solve, we need to introduce some Sobolev
spaces.  Let $L^2(\Omega)$ be the Hilbert space of functions $G:\Omega
\to \C$ such that $\int_\Omega | G(u) |^2 du <\infty$ equipped with
the inner product
\begin{equation*}
\<G,J\>_2 =\int_\Omega G(u) \, {\bar J} (u) \, du\; ,
\end{equation*}
where, for $z\in\C$, $\bar z$ is the complex conjugate of $z$ and
$|z|^2 =z{\bar z}$. The norm of $L^2(\Omega)$ is denoted by $\|
\cdot \|_2$.

Let $H^1(\Omega)$ be the Sobolev space of functions $G$ with
generalized derivatives $\partial_{u_1} G, \dots , \partial_{u_d} G$
in $L^2(\Omega)$. $H^1(\Omega)$ endowed with the scalar product
$\<\cdot, \cdot\>_{1,2}$, defined by
\begin{equation*}
\<G,J\>_{1,2} = \< G, J \>_2 + \sum_{j=1}^d
\<\partial_{u_j} G \, , \, \partial_{u_j} J \>_2\;,
\end{equation*}
is a Hilbert space. The corresponding norm is denoted by
$\|\cdot\|_{1,2}$.

\renewcommand{\labelenumi}{({\bf S\theenumi})}

Let $\varphi:[0,1]\to \bb R_+$ be given by $\varphi (r) = r (1+ ar)$,
let $\nabla \rho$ represent the gradient of some function $\rho$ in
$H^1(\Omega)$: $\nabla \rho=(\partial_{u_1}\rho,\ldots$ ,
$\partial_{u_d}\rho)$, and let $\| \cdot \|$ be the Euclidean norm:
$\|(v_1,\ldots,v_d ) \|^2 =\sum_{1\le i\le d} v_i^2$.  A function
$\rho :\Omega \to [0,1]$ is said to be a weak solution of the elliptic
boundary value problem
\begin{equation}
\label{f01}
 \left\{ \begin{array}{lll}
 \Delta \varphi(\rho) \; =\; 0 \, & \hbox{ on } \Omega\, ,\\ 
\rho \;=\; b & \hbox{ on } \Gamma\, ,
\end{array}
     \right. 
\end{equation}
if
\begin{enumerate}
\item $\rho$ belongs to $H^1(\Omega)$:
\begin{equation*}
\int_\Omega {\parallel\nabla \rho(u)\parallel}^2 
du \;<\; \infty \; .
\end{equation*}

\item For every function $G\in {\cc}^{2}_0\left(\Omega \right)$,
\begin{equation*}
\int_\Omega \big(\Delta G\big)(u) \, \varphi \big(\rho(u)\big) \, du 
=\int_\Gamma \varphi(b(u)) \, \text{\bf n}_1 (u) \,
(\partial_{u_1}  G) (u) \text{d} \text{S}\; ,
\end{equation*}
where {\bf n}=$(\text{\bf n}_1,\ldots ,\text{\bf n}_d)$ stands for the
outward unit normal vector to the boundary surface $\Gamma$ and
$\text{d} \text{S}$ for an element of surface on $\Gamma$.
\end{enumerate}
\smallskip

We prove in Section \ref{sec5} existence and uniqueness of weak
solutions of \eqref{f01}. The first main result of this article
establishes a law of large number for the empirical measure under
$\mss$. Let $\overline\Omega=[-1,1]\times\bb T^{d-1}$ and denote by
$E^\mu$ expectation with respect to a probability measure $\mu$.
Moreover, for a measure $m$ in $\mc M$ and a continuous function
$G:\Omega \to \bb R$, denote by $\<m, G\>$ the integral of $G$ with
respect to $m$:
\begin{equation*}
\<m, G\> \;=\; \int_{\Omega} G(u) \, m(du)\;.
\end{equation*}

\begin{theorem}
\label{th1} For any continuous function $G:\overline\Omega\to\bb R$, 
\begin{equation*}
\lim_{N \rightarrow \infty} E^{\mss}
\Big[\, \Big| \<\pi^N , G \>
-\int_{\overline\Omega} G(u) \bro(u) du \Big| \, \Big] \;=\; 0\, ,
\end{equation*}
where $\bro(u)$ is the unique weak solution of \eqref{f01}.
\end{theorem}

Denote by $\Gamma_{-}$, $\Gamma_{+}$ the left and right boundary of
$\Omega$:
\begin{equation*}
\Gamma_{\pm} = \{ (u_1,\dots, u_d) \in \Omega \ |\ u_1 = \pm 1\}
\end{equation*} 
and denote by $W_{x,x+e_i}$, $x$, $x+e_i \in \Omega_N$, the
instantaneous current over the bond $(x,x+e_i)$. This is the rate at
which a particle jumps from $x$ to $x+e_i$ minus the rate at which a
particle jumps from $x+e_i$ to $x$. A simple computation shows that 
\begin{equation*}
W_{x,x+e_i} \;=\; \{h_{i,x}(\eta) - h_{i,x+e_i}(\eta)\}
\;+\; \{g_{i,x}(\eta) - g_{i,x+2e_i}(\eta)\}
\end{equation*}
provided $x-e_i$ and $x+2e_i$ belongs to $\Omega_N$. Here,
$h_{i,x}(\eta) = \eta(x) - a \eta(x+e_i)\eta(x-e_i)$ and
$g_{i,x}(\eta) = a \eta(x-e_i)\eta(x)$.

\begin{theorem}{\rm (Fick's law)}
\label{th2} 
Fix $-1<u<1$. Then,
\begin{eqnarray*}
\lim_{N\rightarrow \infty} E^{\mss}
\Big[ \frac {2N} { N^{d-1}} \sum_{y\in \bb T^{d-1}_N}
\!\!\!\!\! &\!\!\!\!\! &\!\!\!\!\!  
W_{([uN],y) , ([uN] +1 ,y)} \Big]  \\
&&  
=\int_{\Gamma_-} \varphi(b(v)) \, \text{\rm S}(dv)
-\int_{\Gamma_+} \varphi(b(v)) \, \text{\rm S} (dv)\, .
\end{eqnarray*}
\end{theorem}

\begin{remark}
\label{s01}
We could have considered different bulk dynamics. The important
feature used here to avoid painful arguments is that the process is
gradient, which means that the currents can be written as the
difference of a local function and its translation.
\end{remark}

\subsection{Dynamical large deviations}

Fix $T>0$. Let $\mathcal{M}^0$ be the subset of $\mathcal{M}$ of all
absolutely continuous measures with respect to the Lebesgue measure
with positive density bounded by $1$:

\begin{equation*}
\mc{M}^0=\big\{\pi\in\mc{M}:\pi(du)=\rho(u)du \;\; \hbox{ and } \;\;
0\leq \rho(u)\leq 1\;  \hbox{ a.e.} \big\}\, , 
\end{equation*}
and let $D([0,T],\mc M)$ be the set of right continuous with left
limits trajectories $\pi:[0,T]\to\mc M$, endowed with the Skorohod
topology. $\mathcal{M}^0$ is a closed subset of $\mathcal{M}$ and
$D([0,T],\mc M^0)$ is a closed subset of $D([0,T],\mc M)$.

Let $\Omega_T = (0,T)\times\Omega$ and $\overline{\Omega_T} =
[0,T]\times\overline{\Omega}$. For $1\leq m,n\leq +\infty$, denote
by $\mc C^{m,n}(\overline{\Omega_T})$ the space of functions $G =
G_t(u): \overline{\Omega_T}\to \bb R$ with $m$ continuous derivatives
in time and $n$ continuous derivatives in space. We also denote by
$\mc C^{m,n}_0(\overline{\Omega_T})$ (resp. $\mc
C^{\infty}_c(\Omega_T)$) the set of functions in $\mc
C^{m,n}(\overline{\Omega_T})$ (resp. $\mc
C^{\infty,\infty}(\overline{\Omega_T})$) which vanish at
$[0,T]\times\Gamma$ (resp. with compact support in $\Omega_T$).

Let the energy $\mc Q:D([0,T],\mc M^0)\to[0,\infty]$ be given by
\begin{eqnarray*}
\mc Q(\pi) = \sum_{i=1}^d \sup_{G\in\mc C^{\infty}_c(\Omega_T)}
\Big\{ 2 \int_0^T dt\; \langle \rho_t,\partial_{u_i}G_t\rangle 
- \int_0^Tdt\int_{\Omega} G(t,u)^2\, du \Big\}\, . 
\end{eqnarray*}

For each $G\in\mc C_0^{1,2}(\overline{\Omega_T})$ and each measurable
function $\gamma:\overline\Omega\to[0,1]$, let $\hat J_G = \hat
J_{G,\gamma,T}:D([0,T],\mc M^0)\to\bb R$ be the functional given by
\begin{eqnarray*}
\hat J_{G}(\pi) & = & \langle\pi_T,G_T\rangle-\langle\gamma,G_0\rangle -
\int_0^T \langle \pi_t,\partial_tG_t\rangle \, dt \\
&-& \int_0^T \langle\varphi(\rho_t),\Delta G_t\rangle \, dt
\;+\; \int_0^T dt
\int_{\Gamma^+}\varphi(b) \, \partial_{u_1}G \, dS \\
&-& \int_0^T dt \int_{\Gamma^-} \varphi(b) \, \partial_{u_1}G \, dS \;-\;
\frac{1}{2}\int_0^T \langle\sigma(\rho_t),\Vert\nabla
G_t\Vert^2\rangle \, dt \; , 
\end{eqnarray*}
where $\sigma(r)=2r(1-r)(1+2ar)$ is the mobility and $\pi_t(du) =
\rho_t(u) du$. Define $J_G = J_{G,\gamma,T}:D([0,T],\mc M)\to\bb R$ by
\begin{equation*}
J_G (\pi) =
\begin{cases}
\displaystyle \hat J_G (\pi) & \hbox{ if }  \pi \in D([0,T],\mc M^0) ,\\ 
+\infty & \hbox{ otherwise .}
\end{cases}
\end{equation*}

We define the rate functional
$I_T(\cdot|\gamma):D([0,T],\cm)\to[0,+\infty]$ as
\begin{equation*}
I_T(\pi|\gamma) =
\begin{cases}
\displaystyle \sup_{G\in\cc^{1,2}_0(\overline{\Omega_T})}
\!\big\{J_G(\pi)\big\} & \hbox{ if } \mc Q(\pi)<\infty\, ,\\ 
+\infty & \hbox{ otherwise .}
\end{cases}
\end{equation*}

\begin{theorem} 
\label{mt} 
Fix $T>0$ and a measurable function $\rho_0:\Omega\to[0,1]$. Consider
a sequence $\eta^N$ of configurations in $X_N$ associated to $\rho_0$
in the sense that:
\begin{equation*}
\lim_{N\to\infty} \<\pi^N (\eta^N) , G\> \; =\;
\int_\Omega G(u) \rho_0(u) \, du 
\end{equation*}
for every continuous function $G:\overline{\Omega}\to\bb R$. Then, the
measure $Q_{\eta^N}=\bb P_{\eta^N}(\pi^N)^{-1}$ on $D([0,T],\cm)$
satisfies a large deviation principle with speed $N^d$ and rate
function $I_T(\cdot|\rho_0)$. Namely, for each closed set $\cc\subset
D([0,T],\cm)$,
\begin{equation*}
\limsup_{N\to\infty}\frac{1}{N^d}\log
Q_{\eta^N}(\cc)\leq - \inf_{\pi\in\cc} I_T(\pi|\rho_0)
\end{equation*}
and for each open set $\mc{O}\subset D([0,T],\cm)$,
\begin{equation*}
\liminf_{N\to\infty}\frac{1}{N^d}\log Q_{\eta^N}(\mc{O})\geq -
\inf_{\pi\in\mc{O}} I_T(\pi|\rho_0)\;.
\end{equation*}
Moreover, the rate function $I_T(\cdot|\rho_0)$ is lower
semicontinuous and has compact level sets.
\end{theorem}

\section{Hydrodynamics, Hydrostatics and Fick's law}
\label{sec3}

\renewcommand{\labelenumi}{({\bf H\theenumi})}

We prove in this section Theorem \ref{th1}. The idea is to couple
three copies of the process, the first one starting from the
configuration with all sites empty, the second one starting from the
stationary state and the third one from the configuration with all
sites occupied. The hydrodynamic limit states that the empirical
measure of the first and third copies converge to the solution of the
initial boundary value problem \eqref{f02} with initial condition
equal to $0$ and $1$. Denote these solutions by $\rho^0_t$,
$\rho^1_t$, respectively. In turn, the empirical measure of the second
copy converges to the solution of the same boundary value problem,
denoted by $\rho_t$, with an unknown initial condition.  Since all
solutions are bounded below by $\rho^0$ and bounded above by $\rho^1$,
and since $\rho^j$ converges to a profile $\bar\rho$ as
$t\uparrow\infty$, $\rho_t$ also converges to this profile. However,
since the second copy starts from the stationary state, the
distribution of its empirical measure is independent of time. Hence,
as $\rho_t$ converges to $\bar\rho$, $\rho_0 = \bar\rho$. As we shall
see in the proof, this argument does not require attractiveness of the
underlying interacting particle system.  This approach has been
followed in \cite{MO} to prove hydrostatics for interacting particles
systems with Kac interaction and random potential.

We first describe the hydrodynamic behavior. For a Banach space $(\bb
B,\Vert\cdot\Vert_{\bb B})$ and $T>0$ we denote by $L^2([0,T],\bb B)$
the Banach space of measurable functions $U:[0,T]\to\bb B$ for which
\begin{equation*}
\Vert U\Vert^2_{L^2([0,T],\bb B)} \;=\; 
\int_0^T\Vert U_t\Vert_{\bb B}^2\, dt \;<\; \infty
\end{equation*}
holds. 

Fix $T>0$ and a profile $\rho_0\colon \overline{\Omega} \to [0,1]$. A
measurable function $\rho : [0,T]\times\overline{\Omega} \to [0,1]$ is
said to be a weak solution of the initial boundary value problem
\begin{equation}
\label{f02}
\left\{ 
\begin{array}{l}
\partial_t \rho = \Delta \varphi\big(\rho\big)\, , \\ 
\rho (0 ,\cdot) =\; \rho_0 (\cdot) \, ,   \\
\rho (t, \cdot){\big\vert_\Gamma} =\; b (\cdot) \quad 
\text {for } 0\le t\le T \;,
\end{array}
\right. 
\end{equation}
in the layer $[0,T]\times \Omega$ if

\begin{enumerate}
\item $\rho$ belongs to $L^2 \left( [0,T] , H^1(\Omega)\right)$:
\begin{equation*}
\int_0^T d s \, \Big( \int_\Omega {\parallel\nabla \rho(s,u)\parallel}^2 
du \Big)<\infty \; ;
\end{equation*}

\item For every function $G=G_t(u)$ in
  ${\cc}^{1,2}_0(\overline{\Omega_T})$,
\begin{align*}
& \int_\Omega du \, \big\{ G_T(u)\rho(T,u)-G_0(u)\rho_0 (u)\big\} -
\int_0^T ds \int_\Omega d u \,  (\partial_s G_s)(u)\rho(s,u) \\
& \quad =\; \int_0^T d s \int_\Omega d u \,
(\Delta G_s)(u) \varphi \big(\rho(s,u)\big) 
\; -\; \int_0^T d s \int_\Gamma \varphi(b(u)) {\text{\bf n}}_1 (u) 
(\partial_{u_1} G_s (u)) \text{d} \text{S}  \; .
\end{align*}
\end{enumerate}

We prove in Section \ref{sec5} existence and uniqueness of weak
solutions of \eqref{f02}.  \smallskip

For a measure $\mu$ on $X_N$, denote by $\Pb_\mu = \Pb_\mu^N$ the
probability measure on the path space $D(\R_+, X_N)$ corresponding to
the Markov process $\{\eta_t : t\ge 0\}$ with generator $N^2\cl_N$
starting from $\mu$, and by $\Es_{\mu}$ expectation with respect to
$\Pb_\mu$. Recall the definition of the empirical measure $\pi^N$ and
let $\pi^N_t = \pi^N(\eta_t)$:
\begin{equation*} 
\pi^N_t \, =\, N^{-d}\sum_{x\in\Omega_N}\eta_t(x)\, \delta_{x/N}\; .
\end{equation*}

\begin{theorem}
\label{th3} 
Fix a profile $\rho_0:\Omega\to (0,1)$.  Let $\mu^N$ be a sequence of
measures on $X_N$ associated to $\rho_0$ in the sense that~:
\begin{equation}
\label{f03}
\lim_{N\to\infty} \mu^N \Big\{ \, \Big\vert \<\pi^N, G\> 
-  \int_\Omega G(u) \rho_0(u) \, du \Big\vert >\delta
\Big\}\; =\; 0 \; ,
\end{equation}
for every continuous function $G:\Omega\to\bb R$ and every $\delta
>0$. Then, for every $t>0$,
\begin{equation*}
\lim_{N\to\infty} \Pb_\mu^N \Big\{ \, \Big\vert \<\pi^N_t, G\> 
 -  \int_\Omega G(u) \rho(t,u) \, du \Big\vert 
>\delta \Big\}\; =\; 0 \; ,
\end{equation*}
where $\rho(t,u)$ is the unique weak solution of \eqref{f02}.
\end{theorem}

The proof of this result can be found in \cite{ELS2}.  Denote by
$\Qss^N$ the probability measure on the Skorohod space $D([0,T], \cm)$
induced by the stationary measure $\mss$ and the process
$\{\pi^N(\eta_t): 0\le t\le T\}$. Note that, in contrast with the
usual set-up of hydrodynamics, we do not know that the empirical
measure at time $0$ converges. We can not prove, in particular, that
the sequence $\Qss^N$ converges, but only that this sequence is tight
and that all limit points are concentrated on weak solution of the
hydrodynamic equation for some unknown initial profile.

We first show that the sequence of probability measures $\{\Qss^N :
N\ge 1\}$ is weakly relatively compact:

\begin{proposition}
\label{propo2}
The sequence $\{\Qss^N,\, N\ge 1\}$ is tight and all its limit points
$\Qss^*$ are concentrated on absolutely continuous paths
$\pi(t,du)=\rho(t,u)du$ whose density $\rho$ is positive and bounded
above by $1$~:
\begin{eqnarray*}
\Qss^*\Big\{ \pi\, :\, \pi(t,du)=\rho(t,u)du\, , \hbox{ for } 0\leq
t\leq T\Big\}=1\; , \\
\Qss^*\Big\{ \pi\, :\, 0\le \rho(t,u)\le 1\, , \hbox{ for } (t,u) \in
\overline{\Omega_T}\Big\}=1\; .  
\end{eqnarray*}
\end{proposition}

The proof of this statement is similar to the one of Proposition 3.2
in \cite{LMS} and is thus omitted. Actually, the proof is even simpler
because the model considered here is gradient.

The next two propositions show that all limit points of the sequence
$\{\Qss^N : N\ge 1\}$ are concentrated on absolutely continuous
measures $\pi(t,du)=\rho(t,u)du$ whose density $\rho$ are weak
solution of \eqref{f02} in the layer $[0,T]\times\Omega$. Denote by
${\ca}_T\subset D\big([0,T], \cm^0\big)$ the set of trajectories
$\{\rho(t,u) du: 0\le t\le T\}$ whose density $\rho$
satisfies condition ({\bf H2}) for some initial profile $\rho_0$.

\begin{proposition}
\label{propo1}
All limit points $\Qss^*$ of the sequence $\{\Qss^N,\, N>1\}$ are
concentrated on paths $\pi(t,du)=\rho(t,u)du$ in ${\ca}_T$~:
\begin{equation*}
\Qss^* \{{\ca}_T\}=1\; .
\end{equation*}
\end{proposition}

The proof of this proposition is similar to the one of Proposition 3.3
in \cite{LMS}. Next result states that every limit point $\Qss^*$ of
the sequence $\{\Qss^N,\, N>1\}$ is concentrated on paths whose
density $\rho$ belongs to $L^2([0,T],H^1 (\Omega))$~:

\begin{proposition}
\label{prpo3} 
Let $\Qss^*$ be a limit point of the sequence $\{\Qss^N,\, 
N>1\}$. Then, 
\begin{equation*} 
E_{\Qss^*} \Big[ \int_0^T ds \Big( \int_\Omega \parallel\nabla
\rho(s,u)\parallel^2 d u \Big)\Big] <\infty\; .  
\end{equation*} 
\end{proposition}

The proof of this proposition is similar to the one of Lemma A.1.1 in
\cite{KLO}. We are now ready to prove the first main result of this
article.

\smallskip
\noindent{\bf Proof of Theorem \ref{th1}.}
Fix a continuous function $G:\overline{\Omega} \to \bb R$. We claim that
\begin{equation*}
\lim_{N\to\infty} E^{\mss} \Big[ \, \big| \< \pi , G \> 
- \< \bro(u) du , G\> \big|\, \Big] \;=\;0\;.
\end{equation*}

Note that the expectations are bounded. Consider a subsequence $N_k$
along which the left hand side converges.  It is enough to prove that
the limit vanishes. Fix $T>0$. Since $\mss$ is stationary, by
definition of $\Qss^{N_k}$,
\begin{equation*}
E^{\mu^{N_k}_{\text{ss}}} \Big[ \, \big| \< \pi , G \> - \< \bro(u) du , G\>
\big|\, \Big] \;=\; \Qss^{N_k} \Big[ \, \big| \< \pi_T, G \>
- \< \bro(u) du , G \> \big| \, \Big]\;.
\end{equation*}
Let $\Qss^*$ stand for a limit point of $\{\Qss^{N_k} : k\ge
1\}$. Since the expression inside the expectation is bounded, by
Proposition \ref{propo1},
\begin{eqnarray*}
\lim_{k\to\infty} \Qss^{N_k} \Big[ \big| \< \pi_T , G \>
- \< \bro(u) du , G \> \big| \, \Big]
&=&  \Qss^* \Big[ \big| \< \pi_T , G \>
- \< \bro(u) du , G \> \big| \, \mb 1 \{{\ca}_T\} \Big] \\
&\le& \| G\|_\infty \Qss^* \Big[ \big\| \rho(T,\cdot) -\bro(\cdot) 
\big\|_1 \, \mb 1 \{{\ca}_T\}\Big]\;,
\end{eqnarray*}
where $\Vert\cdot\Vert_1$ stands for the $L^1(\Omega)$ norm.
Denote by $\rho^0(\cdot,\cdot)$ (resp.  $\rho^1(\cdot,\cdot)$) the
weak solution of the boundary value problem \eqref{f02} with initial
condition $\rho(0,\cdot)\equiv 0$ (resp. $\rho(0,\cdot)\equiv 1$). By
Lemma \ref{lembis-ann}, each profile $\rho$ in ${\ca}_T$, including
the stationary profile $\bro$, is bounded below by $\rho^0$ and above
by $\rho^1$.  Therefore
\begin{equation*}
\limsup_{k\to \infty} E^{\mu^{N_k}_{\text{ss}}} \Big[ \, \big| \< \pi , G \> 
- \< \bro(u) du , G\> \big|\, \Big] \;\le\; \| G\|_\infty \; 
\big\| \rho^0(T,\cdot) -\rho^1(T,\cdot) \big\|_1\; .
\end{equation*}
Note that the left hand side does not depend on $T$.  To conclude
the proof it remains to let $T\uparrow \infty$ and to apply
Lemma \ref{s02}.  \cqfd
\smallskip

Fick's law, announced in Theorem \ref{th2}, follows from the
hydrostatics and elementary computations presented in the Proof of
Theorem 2.2 in \cite{KLO}. The arguments here are even simpler and
explicit since the process is gradient.


\section{The rate function $I_T(\cdot | \gamma)$}
\label{sc}


We examine in this section the rate function $I_T(\cdot|\gamma)$. The
main result, presented in Theorem \ref{th4} below, states that
$I_T(\cdot | \gamma)$ has compact level sets. The proof relies on two
ingredients. The first one, stated in Lemma \ref{lem03}, is an
estimate of the energy and of the $H_{-1}$ norm of the time derivative
of a trajectory in terms of the rate function. The second one, stated
in Lemma \ref{lem02}, establishes that sequences of trajectories, with
rate function uniformly bounded, which converges weakly in $L^2$
converge in fact strongly.

We start by introducing some Sobolev spaces. Recall that we denote by
${\mathcal C}_{c}^\infty (\Omega)$ the set of infinitely
differentiable functions $G:\Omega \to \R$, with compact support in
$\Omega$. Recall from subsection $2.1$ the definition of the Sobolev
space $H^1(\Omega)$ and of the norm $\Vert\cdot\Vert_{1,2}$. Denote by
$H^1_0(\Omega)$ the closure of $C_c^{\infty}(\Omega)$ in
$H^1(\Omega)$. Since $\Omega$ is bounded, by Poincar\'e's inequality,
there exists a finite constant $C_1$ such that for all $G\in
H^1_0(\Omega)$
\begin{equation*}
\|G\|^2_2 \;\le\;  C_1 \| \partial_{u_1}G\|^2_2
\;\le\;  C_1 \sum_{j=1}^d \<\partial_{u_j} G \, , \, 
\partial_{u_j} G \>_2 \; .
\end{equation*}
This implies that, in $H^1_0 (\Omega)$
\begin{equation*}
\|G\|_{1,2,0} \;=\; \Big\{ \sum_{j=1}^d
\<\partial_{u_j} G \, , \, \partial_{u_j} G \>_2  \Big\}^{1/2}
\end{equation*}
is a norm equivalent to the norm $\|\cdot \|_{1,2}$.  Moreover, $H^1_0
(\Omega)$ is a Hilbert space with inner product given by
\begin{equation*}
\< G \, , \, J \>_{1,2,0}
\;=\; \sum_{j=1}^d
\<\partial_{u_j} G \, , \, \partial_{u_j} J \>_2 \; .
\end{equation*}

To assign boundary values along the boundary $\Gamma$ of $\Omega$ to
any function $G$ in $H^1(\Omega)$, recall, from the trace Theorem
(\cite{z}, Theorem 21.A.(e)), that there exists a continuous linear
operator $B:H^1(\Omega)\to L^2(\Gamma)$, called trace, such that $BG =
G\big|_{\Gamma}$ if $G\in H^1(\Omega)\cap \mc C(\overline{\Omega})$.
Moreover, the space $H^1_0(\Omega)$ is the space of functions $G$ in
$H^1(\Omega)$ with zero trace (\cite{z}, Appendix (48b)):
\begin{equation*}
H^1_0(\Omega) = \left\{G\in H^1(\Omega):\; BG = 0\right\}\,.
\end{equation*}

Since $\mc C^{\infty}(\overline{\Omega})$ is dense in $H^1(\Omega)$
(\cite{z}, Corollary 21.15.(a)), for functions $F,G$ in $H^1(\Omega)$,
the product $FG$ has generalized derivatives $\partial_{u_i} (FG) =
F\partial_{u_i} G+ G\partial_{u_i} F$ in $L^1(\Omega)$ and
\begin{equation}
\label{ibp}
\begin{split}
& \int_{\Omega}F(u)\, \partial_{u_1}G(u)\, du \; +\; 
\int_{\Omega}G(u) \, \partial_{u_1}F(u) \, du \\
& \quad =\;  \int_{\Gamma_+}
BF(u)\, BG(u)\, du \;-\; \int_{\Gamma_-} BF(u)\, BG(u)\, du\, .
\end{split}
\end{equation}
Moreover, if $G\in H^1(\Omega)$, $f\in\mc C^1(\bb R)$ is such that
$f'$ is bounded, then $f\circ G$ belongs to $H^1(\Omega)$ with
generalized derivatives $\partial_{u_i}(f\circ G) = (f'\circ
G)\partial_{u_i}G$ and trace $B(f\circ G) = f\circ(BG)$.

Finally, denote by $H^{-1}(\Omega)$ the dual of $H^1_0(\Omega)$.
$H^{-1}(\Omega)$ is a Banach space with norm $\Vert\cdot\Vert_{-1}$
given by
\begin{equation*}
\Vert v\Vert^2_{-1} = \sup_{G\in\mc C^{\infty}_c(\Omega)}
\left\{2\langle v,G\rangle_{-1,1} -
\int_{\Omega} \Vert \nabla G(u)\Vert^2du \right\}\, , 
\end{equation*}
where $\langle v,G\rangle_{-1,1}$ stands for the values of the linear
form $v$ at $G$.

For each $G\in\mc C^{\infty}_c(\Omega_T)$ and each integer $1\leq
i\leq d$, let $\mc Q_i^G:D([0,T],\mc M^0)\to\bb R$ be the functional
given by 
\begin{eqnarray*}
\mc Q_i^G(\pi) = 2\int_0^Tdt\;\langle\pi_t,\partial_{u_i}G_t\rangle -
\int_0^Tdt\int_{\Omega} du\; G(t,u)^2\, , 
\end{eqnarray*}
and recall, from subsection 2.2, that the energy $\mc Q(\pi)$ was
defined as
\begin{equation*}
\mc Q(\pi) = \sum_{i=1}^d\mc Q_i(\pi)\;\; \hbox{ with } \;\;\mc
Q_i(\pi) = \sup_{G\in\mc C_c^{\infty}(\Omega_T)}{\mc Q_i^G(\pi)}\, . 
\end{equation*}

The functional $\mc Q^G_i$ is convex and continuous in the Skorohod
topology. Therefore $\mc Q_i$ and $\mc Q$ are convex and lower
semicontinuous. Furthermore, it is well known that a measure
$\pi(t,du) = \rho(t,u) du$ in $D([0,T], \mc M)$ has finite energy,
$\mc Q(\pi) < \infty$, if and only if its density $\rho$ belongs to
$L^2([0,T] , H^1(\Omega))$, in which case
\begin{equation*}
\hat{\mc Q}(\pi) \;:=\;
\int_0^Tdt\int_{\Omega}du\;\Vert\nabla\rho_t(u)\Vert^2 \;<\; \infty
\end{equation*}
and $\mc Q(\pi) = \hat{\mc Q}(\pi)$.

Let $D_{\gamma} = D_{\gamma,b}$ be the subset of $C([0,T],\mc M^0)$
consisting of all paths $\pi(t,du) = \rho(t,u) du$ with initial
profile $\rho(0,\cdot) = \gamma(\cdot)$, finite energy $\mc Q(\pi)$ (in
which case $\rho_t$ belongs to $H^1(\Omega)$ for almost all $0\leq
t\leq T$ and so $B(\rho_t)$ is well defined for those $t$) and such
that $B(\rho_t) = b$ for almost all $t$ in $[0,T]$.

\begin{lemma}
\label{lem01}
Let $\pi$ be a trajectory in $D([0,T],\mc M)$ such that
$I_T(\pi|\gamma)<\infty$. Then $\pi$ belongs to $D_{\gamma}$.
\end{lemma}

\begin{proof}
Fix a path $\pi$ in $D([0,T],\mc M)$ with finite rate function,
$I_T(\pi|\gamma)<\infty$. By definition of $I_T$, $\pi$ belongs to
$D([0,T],\mc M^0)$. Denote its density by $\rho$: $\pi(t,du) =
\rho(t,u) du$. 

The proof that $\rho(0,\cdot) = \gamma(\cdot)$ is similar to the one
of Lemma 3.5 in \cite{BDGJL}.  To prove that $B(\rho_t) = b$ for
almost all $t\in[0,T]$, since the function $\varphi:[0,1]\to[0,1+a]$
is a $\mc C^1$ diffeomorphism and since $B(\varphi\circ\rho_t) =
\varphi(B\rho_t)$ (for those $t$ such that $\rho_t$ belongs to
$H^1(\Omega)$), it is enough to show that $B(\varphi\circ\rho_t) =
\varphi(b)$ for almost all $t\in[0,T]$. To this end, we just need to
show that, for any function $H_{\pm}\in\mc
C^{1,2}([0,T]\times\Gamma_{\pm})$,
\begin{eqnarray}
\label{trb}
\int_0^T dt\int_{\Gamma_{\pm}}
du\; \big\{ B(\varphi(\rho_t))(u)-\varphi(b(u)) \big\}
\, H_{\pm}(t,u) \;=\; 0\, . 
\end{eqnarray}
Fix a function $H\in\mc C^{1,2}([0,T]\times\Gamma_{-})$. For each
$0<\theta<1$, let $h_{\theta}:[-1,1]\to\bb R$ be the function given by
\begin{equation*}
h_{\theta}(r) = 
\begin{cases}
r+1 & \hbox{ if } -1\leq r\leq-1+\theta\, , \\
\frac{-\theta r}{1-\theta} & \hbox{ if } -1+\theta\leq r\leq 0\, , \\
0 & \hbox{ if } 0\leq r\leq 1\, ,
\end{cases}
\end{equation*}
and define the function $G_{\theta}:\overline{\Omega_T}\to \bb R$ as
$G(t,(u_1,\check{u})) = h_{\theta}(u_1)H(t,(-1,\check{u}))$ for all
$\check{u}\in\bb T^{d-1}$. Of course, $G_{\theta}$ can be approximated
by functions in $\mc C^{1,2}_0(\overline{\Omega_T})$. From the
integration by parts formula \eqref{ibp} and the definition of
$J_{G_{\theta}}$, we obtain that
\begin{equation*}
\lim_{\theta\to 0}J_{G_{\theta}}(\pi) = \int_0^T
dt\int_{\Gamma_-}\!\!\!
du\; \big\{ B(\varphi(\rho_t))(u)-\varphi(b(u)) \big\}\, H(t,u)\, ,
\end{equation*}
which proves \eqref{trb} because $I_T(\pi|\gamma)<\infty$.

We deal now with the continuity of $\pi$. We claim that there exists a
positive constant $C_0$ such that, for any $g\in\mc
C^{\infty}_c(\Omega)$, and any $0 \leq s < r < T$,
\begin{eqnarray}\label{cont}
|\langle\pi_r,g\rangle-\langle\pi_s,g\rangle| \;\leq\;
C_0(r-s)^{1/2}\left\{I_T(\pi|\gamma)+\Vert
  g\Vert^2_{1,2,0}+(r-s)^{1/2}\Vert\Delta g\Vert_1\right\}\, . 
\end{eqnarray}
Indeed, for each $\delta>0$, let $\psi^{\delta}:[0,T]\to\bb R$ be the
function given by
\begin{equation*}
(r-s)^{1/2}\psi^{\delta}(t) = 
\begin{cases}
0 & \hbox{ if } 0\leq t\leq s \;\hbox{ or }\; r+\delta\leq t\leq T\, , \\
\frac{t-s}{\delta} & \hbox{ if } s\leq t\leq s+\delta\, , \\
1 & \hbox{ if } s+\delta\leq t\leq r\, , \\
1-\frac{t-r}{\delta} & \hbox{ if } r\leq t\leq r+\delta\, ,
\end{cases}
\end{equation*}
and let $G^{\delta}(t,u) = \psi^{\delta}(t)g(u)$. Of course,
$G^{\delta}$ can be approximated by functions in $\mc
C^{1,2}_0(\overline{\Omega_T})$ and then
\begin{eqnarray*}
(r-s)^{1/2}\lim_{\delta\to 0}J_{G^{\delta}}(\pi) & = &
\langle\pi_r,g\rangle-\langle\pi_s,g\rangle -
\int_s^rdt\;\langle\varphi(\rho_t),\Delta g\rangle \\
&-& \frac{1}{2(r-s)^{1/2}}\int_s^r dt\;\langle\sigma(\rho_t),
\Vert\nabla g\Vert^2\rangle\, . 
\end{eqnarray*}
To conclude the proof, it remains to observe that the left hand side
is bounded by $(r-s)^{1/2} I_T(\pi|\gamma)$, and to note that
$\varphi$, $\sigma$ are positive and bounded above on $[0,1]$ by some
positive constant.
\end{proof}

Denote by $L^2([0,T],H_0^1(\Omega))^*$ the dual of
$L^2([0,T],H_0^1(\Omega))$.  By Proposition 23.7 in \cite{z},
$L^2([0,T],H_0^1(\Omega))^*$ corresponds to
$L^2([0,T],H^{-1}(\Omega))$ and for $v$ in
$L^2([0,T],H_0^1(\Omega))^*$, $G$ in $L^2([0,T],H_0^1(\Omega))$,
\begin{equation}
\label{fl1}
\<\!\< v,G \>\!\>_{-1,1} \; =\; \int_0^T \<v_t, G_t\>_{-1,1}\, dt \; , 
\end{equation}
where the left hand side stands for the value of the linear functional
$v$ at $G$. Moreover, if we denote by $|\!| \!| v |\!| \!|_{-1}$ the
norm of $v$,
\begin{equation*}
|\!| \!| v |\!| \!|^2_{-1} \;=\; \int_0^T \Vert v_t \Vert^2_{-1} \, dt\;.
\end{equation*}

Fix a path $\pi(t,du)=\rho(t,u)du$ in $D_\gamma$ and suppose that
\begin{equation}
\label{cl1}
\sup_{H\in\mc C^{\infty}_c(\Omega_T)}\Big \{2 \int_0^T
  dt\, \langle\rho_t,\partial_tH_t\rangle - \int_0^T dt\int_{\Omega}
  du\; \Vert\nabla H_t\Vert^2\Big\}\;<\; \infty\;.
\end{equation} 
In this case $\partial_t \rho : C^{\infty}_c(\Omega_T) \to \bb R$
defined by
\begin{equation*}
\partial_t \rho (H) \;=\; - \int_0^T \< \rho_t, \partial_t H_t\>\, dt
\end{equation*}
can be extended to a bounded linear operator $\partial_t \rho :
L^2([0,T],H_0^1(\Omega)) \to \bb R$. It belongs therefore to
$L^2([0,T],H_0^1(\Omega))^* = L^2([0,T],H^{-1}(\Omega))$. In
particular, there exists $v = \{v_t :0\le t\le T\}$ in
$L^2([0,T],H^{-1}(\Omega))$, which we denote by $v_t = \partial_t
\rho_t$, such that for any $H$ in $L^2([0,T],H_0^1(\Omega))$,
\begin{equation*}
\<\!\< \partial_t\rho , H\>\!\>_{-1,1} \;=\;
\int_0^T \langle \partial_t \rho_t , H_t\rangle_{-1,1}\,dt \;.
\end{equation*}
Moreover,
\begin{equation*}
\begin{split}
|\!| \!| \partial_t \rho |\!| \!|^2_{-1} \; &=\;
\int_0^T \, \Vert \partial_t \rho_t \Vert^2_{-1} \, dt\\
&=\; \sup_{H\in\mc C^{\infty}_c(\Omega_T)}\Big \{2 \int_0^T
  dt\, \langle\rho_t,\partial_tH_t\rangle - \int_0^T dt\int_{\Omega}
  du\; \Vert\nabla H_t\Vert^2\Big\} \;.
\end{split}
\end{equation*}

Let $W$ be the set of paths $\pi(t,du)=\rho(t,u)du$ in $D_\gamma$ such
that \eqref{cl1} holds, i.e., such that $\partial_t\rho$ belongs to
$L^2\left([0,T], H^{-1}(\Omega)\right)$.  For $G$ in
$L^2\left([0,T],H^1_0(\Omega)\right)$, let $\bb J_{G}:W\to \bb R$ be
the functional given by
\begin{eqnarray*}
\bb J_{G}(\pi) & = & 
\<\!\< \partial_t\rho , G\>\!\>_{-1,1} + \int_0^T
dt\int_{\Omega} du\; \nabla G_t(u)\cdot\nabla(\varphi(\rho_t(u)))  \\ 
&-& \frac{1}{2}\int_0^T dt\int_{\Omega} du\;
\sigma(\rho_t(u))\, \Vert\nabla G_t(u)\Vert^2\, . 
\end{eqnarray*}
Note that $\bb J_{G}(\pi) = J_G(\pi)$ for every $G$ in $\mc
C^{\infty}_c(\Omega_T)$. Moreover, since $\bb J_{\cdot} (\pi)$ is
continuous in $L^2\left([0,T],H^1_0 (\Omega)\right)$ and since $\mc
C^{\infty}_c(\Omega_T)$ is dense in $\mc C^{1,2}_0
(\overline{\Omega_T})$ and in $L^2 ([0,T]$, $H^1_0(\Omega))$, for
every $\pi$ in $W$,
\begin{eqnarray}
\label{rf1}
I_T(\pi|\gamma) = \sup_{G\in \mc C^{\infty}_c(\Omega_T)} 
\bb J_{G}(\pi) \;=\;
\sup_{G\in L^2\left([0,T],H^1_0\right)} \bb J_{G}(\pi)\, . 
\end{eqnarray}

\begin{lemma}
\label{lem03}
There exists a constant $C_0>0$ such that if the density $\rho$ of
some path $\pi (t,du) = \rho(t,u) du$ in $D([0,T],\mc M^0)$ has a
generalized gradient, $\nabla\rho$, then
\begin{eqnarray}
\label{est1}
\int_0^Tdt\;\left\Vert\partial_t\rho_t\right\Vert_{-1}^2 &\leq&
C_0\left\{I_T(\pi|\gamma)+\mc Q(\pi)\right\}\, , \\
\label{est2}
\int_0^T dt \int_{\Omega}
du\;\frac{\Vert\nabla\rho_t(u)\Vert^2}{\chi(\rho_t(u))} &\leq&
C_0\left\{I_T(\pi|\gamma) +1\right\}\, , 
\end{eqnarray}
where $\chi(r) =r(1-r)$ is the static compressibility.
\end{lemma}

\begin{proof}
Fix a path $\pi(t,du) = \rho(t,u) du$ in $D([0,T],\mc M^0)$.  In view
of the discussion presented before the lemma, we need to show that the
left hand side of \eqref{cl1} is bounded by the right hand side of
\eqref{est1}.  Such an estimate follows from the definition of the
rate function $I_T(\cdot|\gamma)$ and from the elementary inequality
$2ab\leq Aa^2+A^{-1}b^2$.

We turn now to the proof of \eqref{est2}. We may of course assume that
$I_T(\pi|\gamma)<\infty$, in which case $\mc Q(\pi)<\infty$. Fix a
function $\beta$ as in the beginning of Section \ref{sec2}.  For each
$\delta>0$, let $h^{\delta}:[0,1]^2\to\bb R$ be the function given by
\begin{equation*}
h^{\delta}(x,y) = (x+\delta)\log\left(\frac{x+\delta}{y+\delta}\right)
+ (1-x+\delta)\log\left(\frac{1-x+\delta}{1-y+\delta}\right)\, . 
\end{equation*}

By \eqref{est1}, $\partial_t\rho$ belongs to
$L^2([0,T],H^{-1}(\Omega))$. We claim that
\begin{eqnarray}
\label{ibp0}
\int_0^Tdt\;\langle\partial_t\rho_t,\partial_xh^{\delta}
(\rho_t,\beta)\rangle_{-1,1}  
& = & \int_{\Omega}h^{\delta}(\rho_{_T}(u),\beta(u))du\nonumber  \\
& & - \int_{\Omega}h^{\delta}(\rho_0(u),\beta(u))du\, .
\end{eqnarray}

Indeed, By Lemma \ref{lem01} and \eqref{est1}, $\rho-\beta$ belongs to
$L^2\left([0,T],H^1_0(\Omega)\right)$ and $\partial_t(\rho-\beta) =
\partial_t\rho$ belongs to $L^2([0,T],H^{-1}(\Omega))$. Then, there
exists a sequence $\{\widetilde G^n: \,n\geq 1\}$ of smooth functions
$\widetilde G^n:\overline{\Omega_T}\to \bb R$ such that $\widetilde
G^n_t$ belongs to $\mc C^{\infty}_c(\Omega)$ for every $t$ in $[0,T]$,
$\widetilde G^n$ converges to $\rho-\beta$ in
$L^2([0,T],H^1_0(\Omega))$ and $\partial_t \widetilde G^n$ converges
to $\partial_t(\rho-\beta)$ in $L^2([0,T],H^{-1}(\Omega))$ (cf.
\cite{z}, Proposition 23.23(ii)). For each positive integer $n$, let
$G^n = \widetilde G^n+\beta$ and for each $\delta>0$, fix a smooth
function $\tilde h^{\delta} : \bb R^2\to\bb R$ with compact support
and such that its restriction to $[0,1]^2$ is $h^{\delta}$. It is
clear that
\begin{eqnarray}
\label{ibpn}
\int_0^Tdt\;\langle\partial_tG^n_t, \partial_x\tilde
h^{\delta}(G^n_t,\beta)\rangle & = & \int_{\Omega}\tilde
h^{\delta}(G^n_T(u),\beta(u))du\nonumber \\
& & - \int_{\Omega}\tilde h^{\delta}(G^n_0(u),\beta(u))du\, .
\end{eqnarray}

On the one hand, $\partial_xh^{\delta}:[0,1]^2\to\bb R$ is given by
\begin{equation*}
\partial_xh^{\delta}(x,y) =
\log\left(\frac{x+\delta}{1-x+\delta}\right) -
\log\left(\frac{y+\delta}{1-y+\delta}\right)\, . 
\end{equation*}
Hence, $\partial_xh^{\delta}(\rho,\beta)$ and $\partial_x\tilde
h^{\delta}(G^n,\beta)$ belongs to
$L^2\left([0,T],H^1_0(\Omega)\right)$. Moreover, since $\partial_x
\tilde h^{\delta}$ is smooth with compact support and $G^n$ converges
to $\rho$ in $L^2([0,T],H^1(\Omega))$, $\partial_x\tilde
h^{\delta}(G^n,\beta)$ converges to $\partial_xh^{\delta}(\rho,\beta)$
in $L^2([0,T],H^1_0(\Omega))$. From this fact and since
$\partial_tG^n$ converges to $\partial_t\rho$ in
$L^2([0,T],H^{-1}(\Omega))$, if we let $n\to\infty$, the left hand
side in \eqref{ibpn} converges to
\begin{equation*}
\int_0^Tdt\;\langle\partial_t\rho_t,
\partial_xh^{\delta}(\rho_t,\beta)\rangle_{-1,1}\, . 
\end{equation*}

On the other hand, by Proposition 23.23(ii) in \cite{z}, $G^n_0$,
resp. $G^n_T$, converges to $\rho_0$, resp. $\rho_T$, in $L^2(\Omega)$.
Then, if we let $n\to\infty$, the right hand side in \eqref{ibpn} goes
to
\begin{equation*}
\int_{\Omega}h^{\delta}(\rho_{_T}(u),\beta(u))du -
\int_{\Omega}h^{\delta}(\rho_0(u),\beta(u))du\, , 
\end{equation*}
which proves claim \eqref{ibp0}.

Notice that, since $\beta$ is bounded away from 0 and 1, there exists
a positive constant $C=C(\beta)$ such that for $\delta$ small enough,
\begin{eqnarray}\label{bou1}
h^{\delta}(\rho(t,u),\beta(u))\leq C \;\hbox{ for all }\; (t,u)\;\hbox{
  in } \;\overline{\Omega_T}\, . 
\end{eqnarray}

For each $\delta>0$, let $H^{\delta}:\overline{\Omega_T}\to\bb R$ be
the function given by
\begin{equation*}
H^{\delta}(t,u) =
\frac{\partial_xh^{\delta}(\rho(t,u),\beta(u))}{2(1+2\delta)}\, . 
\end{equation*}
A simple computation shows that
\begin{eqnarray*}
\bb J_{H^{\delta}}(\pi) & \geq & \int_0^T
dt\left\langle\partial_t\rho_t,H^{\delta}_t\right\rangle_{-1,1}
+\frac{1}{4}\int_0^T dt \int_{\Omega}du\;
\varphi'(\rho_t(u))\frac{\Vert\nabla\rho_t(u)\Vert^2}{\chi_{\delta}(\rho_t(u))}  
\\ & & - \frac{1}{8}\int_0^T dt
\int_{\Omega}du\;\sigma_{\delta}(\rho_t(u))
\frac{\Vert\nabla\beta(u)\Vert^2}{{\chi_{\delta}(\beta(u))}^2}\,, 
\end{eqnarray*}
where $\chi_{\delta}(r) = (r+\delta)(1-r+\delta)$ and
$\sigma_{\delta}(r)=2\chi_{\delta}(r)\varphi'(r)$. This last
inequality together with \eqref{ibp0}, \eqref{rf1} and \eqref{bou1}
show that there exists a positive constant $C_0=C_0(\beta)$ such that
for $\delta$ small enough
\begin{equation*}
C_0 \left\{I_T(\pi|\gamma) + 1\right\} \geq \int_0^T dt \int_{\Omega}
du\; \frac{\Vert\nabla\rho(t,u)\Vert^2}{\chi_{\delta}(\rho(t,u))}\,. 
\end{equation*}
We conclude the proof by letting $\delta\downarrow 0$ and by using
Fatou's lemma. 
\end{proof}

\begin{corollary}
\label{rfhe}
The density $\rho$ of a path $\pi(t,du)=\rho(t,u)du$ in $D([0,T],\mc
M^0)$ is the weak solution of the equation \eqref{f02} with initial
profile $\gamma$ if and only if the rate function $I_T(\pi|\gamma)$
vanishes. Moreover, in that case
\begin{equation*}
\int_0^T dt \int_{\Omega}
du\;\frac{\Vert\nabla\rho_t(u)\Vert^2}{\chi(\rho_t(u))}
\;<\;\infty\;. 
\end{equation*}
\end{corollary}

\begin{proof}
  On the one hand, if the density $\rho$ of a path $\pi(t,du)=
  \rho(t,u)du$ in $D([0,T], \mc M^0)$ is the weak solution of equation
  \eqref{f02}, by assumption ({\bf H1}), the energy $Q(\pi)$ is
  finite. Moreover, since the initial condition is $\gamma$, in the
  formula of $\hat J_G(\pi)$, the linear part in $G$ vanishes which
  proves that the rate functional $I_T(\pi|\gamma)$ vanishes. On the
  other hand, if the rate functional vanishes, the path $\rho$ belongs
  to $L^2([0,T],H^1(\Omega))$ and the linear part in $G$ of $J_G(\pi)$
  has to vanish for all functions $G$. In particular, $\rho$ is a weak
  solution of \eqref{f02}. Moreover, in that case, by the previous
  lemma, the bound claimed holds.
\end{proof}

For each $q>0$, let $E_q$ be the level set of $I_T(\pi|\gamma)$
defined by
\begin{equation*}
E_q=\left\{\pi\in D([0,T],\mc M): I_T(\pi|\gamma)\leq q\right\}\, .
\end{equation*}
By Lemma \ref{lem01}, $E_q$ is a subset of $C([0,T],\mc M^0)$. Thus,
from the previous lemma, it is easy to deduce the next result.

\begin{corollary}
\label{corls}
For every $q\geq 0$, there exists a finite constant $C(q)$ such that
\begin{eqnarray*}
\sup_{\pi\in E_q} \Big\{ \int_0^T dt\;
\left\Vert\partial_t\rho_t\right\Vert_{-1}^2
\;+\; \int_0^T dt \int_{\Omega} du\;
\frac{\Vert\nabla\rho(t,u)\Vert^2}{\chi(\rho(t,u))}
\Big \} \;\leq\; C(q)\;. 
\end{eqnarray*}
\end{corollary}

Next result together with the previous estimates provide the
compactness needed in the proof of the lower semicontinuity of the
rate function.

\begin{lemma}
\label{lem02}
Let $\{\rho^n:n\geq 1\}$ be a sequence of functions in $L^2(\Omega_T)$
such that uniformly on $n$, 
\begin{equation*}
\int_0^T dt\left\Vert\rho^n_t\right\Vert^2_{1,2} + \int_0^T dt
\left\Vert\partial_t\rho^n_t\right\Vert_{-1}^2 < C 
\end{equation*}
for some positive constant $C$. Suppose that $\rho\in L^2(\Omega_T)$
and that $\rho^n \rightarrow \rho$ weakly in $L^2(\Omega_T)$. Then
$\rho_n\rightarrow \rho$ strongly in $L^2(\Omega_T)$.
\end{lemma}

\begin{proof}
Since $H^1(\Omega)\subset L^2(\Omega)\subset H^{-1}(\Omega)$ with
compact embedding $H^1(\Omega)\to L^2(\Omega)$, from Corollary 8.4,
\cite{Si}, the sequence $\{\rho_n\}$ is relatively compact in
$L^2\big([0,T],L^2(\Omega)\big)$. Therefore the weak convergence
implies the strong convergence in $L^2\big([0,T],L^2(\Omega)\big)$.
\end{proof}

\begin{theorem}
\label{th4}
The functional $I_T(\cdot|\gamma)$ is lower semicontinuous and has
compact level sets. 
\end{theorem}

\begin{proof}
We have to show that, for all $q\geq 0$, $E_q$ is compact in
$D([0,T],\mc M)$. Since $E_q\subset C([0,T],\mc M^0)$ and $C([0,T],\mc
M^0)$ is a closed subset of $D([0,T],\mc M)$, we just need to show
that $E_q$ is compact in $C([0,T],\mc M^0)$.

We will show first that $E_q$ is closed in $C([0,T],\mc M^0)$. Fix
$q\in\bb{R}$ and let $\{\pi^n:\,n\geq 1\}$ be a sequence in $E_q$
converging to some $\pi$ in $C([0,T],\mc M^0)$. Then, for all $G\in\mc
C(\overline{\Omega_T})$,
\begin{eqnarray*}
\lim_{n\to\infty}\int_0^T dt\;\langle\pi^n_t,G_t\rangle = \int_0^T
dt\;\langle\pi_t,G_t\rangle\, . 
\end{eqnarray*}
Notice that this means that $\pi^n\rightarrow\pi$ weakly in
$L^2(\Omega_T)$, which together with Corollary \ref{corls} and Lemma
\ref{lem02} imply that $\pi^n\rightarrow\pi$ strongly in
$L^2(\Omega_T)$. From this fact and the definition of $J_G$ it is easy
to see that, for all $G$ in $\mc C_0^{1,2}(\overline{\Omega_T})$,
\begin{equation*}
\lim_{n\to\infty}J_G(\pi_n) = J_G(\pi)\, .
\end{equation*}
This limit, Corollary \ref{corls} and the lower semicontinuity of $\mc
Q$ permit us to conclude that $\mc Q(\pi)\leq C(q)$ and that
$I_T(\pi|\gamma)\leq q$.

We prove now that $E_q$ is relatively compact. To this end, it is
enough to prove that for every continuous function
$G:\overline\Omega\to\bb R$,
\begin{eqnarray}
\label{est4}
\lim_{\delta\to 0}\sup_{\pi\in E_q}\sup_{\substack{0\leq s,r\leq T\\
  |r-s|<\delta}}|\langle\pi_r,G\rangle-\langle\pi_s,G\rangle|=0\, . 
\end{eqnarray}
Since $E_q\subset C([0,T],\mc M^0)$, we may assume by approximations
of $G$ in $L^1(\Omega)$ that $G\in\mc C_c^{\infty}(\Omega)$. In which
case, \eqref{est4} follows from \eqref{cont}.
\end{proof}

We conclude this section with an explicit formula for the rate
function $I_T(\cdot |\gamma)$.  For each $\pi (t,du) = \rho(t,u)du$ in
$D([0,T], \mc M^0)$, denote by $H^1_0(\sigma(\rho))$ the Hilbert space
induced by $\mc C^{1,2}_0(\overline{\Omega_T})$ endowed with the inner
product $\langle\cdot,\cdot\rangle_{\sigma(\rho)}$ defined by
\begin{equation*}
\langle H,G\rangle_{\sigma(\rho)}
=\int_0^Tdt\; \langle\sigma(\rho_t),\nabla H_t\cdot\nabla G_t\rangle
\,. 
\end{equation*}
Induced means that we first declare two functions $F,G$ in $\mc
C^{1,2}_0(\overline{\Omega_T})$ to be equivalent if $\langle
F-G,F-G\rangle_{\sigma(\rho)} = 0$ and then we complete the quotient
space with respect to the inner product
$\langle\cdot,\cdot\rangle_{\sigma(\rho)}$. The norm of
$H^1_0(\sigma(\rho))$ is denoted by $\Vert\cdot\Vert_{\sigma(\rho)}$.

Fix a path $\rho$ in $D([0,T], \mc M^0)$ and a function $H$ in
$H^1_0(\sigma(\rho))$.  A measurable function $\lambda :
[0,T]\times\Omega \to [0,1]$ is said to be a weak solution of the
nonlinear boundary value parabolic equation
\begin{eqnarray}
\label{f05}
\begin{cases}
\partial_t\lambda   & = \;\;  \Delta\varphi(\lambda) -
\sum_{i=1}^d\partial_{u_i} \left(\sigma(\lambda)
\partial_{u_i} H\right)\, ,\\ 
\lambda(0,\cdot)    & =  \;\; \gamma\, ,\\
\lambda(t,\cdot)|_\Gamma   & = \;\; b \;\;\; 
\hbox{ for }\; 0\leq t\leq T \, .
\end{cases}
\end{eqnarray}
if it satisfies the following two conditions.

\begin{enumerate}
\item[\bf (H1')] $\lambda$ belongs to $L^2 \left( [0,T] ,
    H^1(\Omega)\right)$: 
\begin{equation*}
\int_0^T d s\Big( \int_\Omega {\parallel\nabla \lambda(s,u)\parallel}^2 
du \Big)<\infty \; ;
\end{equation*}

\item[\bf (H2')] For every function $G(t,u)=G_t(u)$ in ${\cc}^{1,2}_0
  (\overline{\Omega_T})$,
\begin{align*}
& \int_\Omega du \, \big\{ G_T(u)\rho(T,u)-G_0(u)\gamma (u)\big\} -
\int_0^T ds \int_\Omega d u \,  (\partial_s G_s)(u)\lambda(s,u) \\
& \quad =\; \int_0^T d s \int_\Omega d u \,
(\Delta G_s)(u) \varphi \big(\lambda(s,u)\big) 
\; -\; \int_0^T ds \int_\Gamma \varphi(b(u)) {\text{\bf n}}_1 (u) 
(\partial_{u_1} G_s (u)) \text{d} \text{S} \\
& \quad+\int_0^T ds \int_{\Omega} du\; \sigma(\lambda(s,u))\nabla
H_s(u)\cdot\nabla G_s(u)\; . 
\end{align*}
\end{enumerate}

In Section \ref{sec5} we prove uniqueness of weak solutions of
equation \eqref{f05} when $H$ belongs to $L^2 \left( [0,T] ,
  H^1(\Omega)\right)$, i.e., provided
\begin{equation*}
\int_0^T dt\int_{\Omega} du\; \Vert\nabla H_t(u)\Vert^2 < \infty\, .
\end{equation*}

\begin{lemma}
\label{lem05}
Assume that $\pi(t,du) = \rho(t,u)du$ in $D([0,T],\mc M^0)$ has finite
rate function: $I_T(\pi|\gamma)<\infty$. Then, there exists a function
$H$ in $H^1_0(\sigma(\rho))$ such that $\rho$ is a weak solution to
\eqref{f05}.  Moreover,
\begin{eqnarray}
\label{f06}
I_T(\pi|\gamma) = \frac{1}{2}\Vert H\Vert_{\sigma(\rho)}^2\, .
\end{eqnarray}
\end{lemma}

The proof of this lemma is similar to the one of Lemma 5.3 in
\cite{KL} and is therefore omitted.

\section{$I_T(\cdot|\gamma)$-Density}

The main result of this section, stated in Theorem \ref{th5}, asserts
that any trajectory $\lambda_t$, $0\le t\le T$, with finite rate
function, $I_T(\lambda|\gamma)<\infty$, can be approximated by a
sequence of smooth trajectories $\{\lambda^n : n\ge 1\}$ such that
\begin{equation*}
\lambda^n \longrightarrow \lambda \quad\text{and}\quad 
I_T(\lambda^n|\gamma)  \longrightarrow  I_T(\lambda|\gamma)\;.
\end{equation*}
This is one of the main steps in the proof of the lower bound of the
large deviations principle for the empirical measure.  The proof
reposes mainly on the regularizing effects of the hydrodynamic
equation and is one of the main contributions of this article, since
it simplifies considerably the existing methods.

A subset $A$ of $D([0,T],\mc M)$ is said to be
$I_T(\cdot|\gamma)$-dense if for every $\pi$ in $D([0,T],\mc M)$ such
that $I_T(\pi|\gamma)<\infty$, there exists a sequence $\{\pi^n : n\ge
1\}$ in $A$ such that $\pi^n$ converges to $\pi$ and
$I_T(\pi^n|\gamma)$ converges to $I_T(\pi|\gamma)$.

Let $\Pi_1$ be the subset of $D([0,T],\mc M^0)$ consisting of paths
$\pi(t,du) = \rho(t,u)du$ whose density $\rho$ is a weak solution of
the hydrodynamic equation \eqref{f02} in the time interval
$[0,\delta]$ for some $\delta>0$.

\begin{lemma}
The set $\Pi_1$ is $I_T(\cdot|\gamma)$-dense.
\end{lemma}

\begin{proof}
Fix $\pi(t,du) = \rho(t,u) du$ in $D([0,T],\mc M)$ such that
$I_T(\pi|\gamma)<\infty$. By Lemma \ref{lem01}, $\pi$ belongs to
$C([0,T],\mc M^0)$. For each $\delta>0$, let $\rho^{\delta}$ be the
path defined as
\begin{equation*}
\rho^{\delta}(t,u) = 
\begin{cases}
\lambda(t,u) & \hbox{ if } 0\leq t\leq\delta\, , \\
\lambda(2\delta-t,u) & \hbox{ if } \delta\leq t\leq 2\delta\, , \\
\rho(t-2\delta,u) & \hbox{ if } 2\delta\leq t\leq T\, ,
\end{cases}
\end{equation*}
where $\lambda$ is the weak solution of the hydrodynamic equation
\eqref{f02} starting at $\gamma$.  It is clear that $\pi^\delta(t,du)
= \rho^{\delta}(t,u)du$ belongs to $D_{\gamma}$, because so do $\pi$
and $\lambda$ and that $\mc Q(\pi^{\delta})\leq \mc Q(\pi) +2\mc
Q(\lambda)<\infty$. Moreover, $\pi^{\delta}$ converges to $\pi$ as
$\delta\downarrow 0$ because $\pi$ belongs to $\mc C([0,T],\mc M)$. By
the lower semicontinuity of $I_T(\cdot|\gamma)$, $I_T(\pi|\gamma)\leq
\liminf_{\delta\to 0} I_T(\pi^{\delta}|\gamma)$. Then, in order to
prove the lemma, it is enough to prove that $I_T(\pi|\gamma)\geq
\limsup_{\delta\to 0} I_T(\pi^{\delta}|\gamma)$. To this end,
decompose the rate function $I_T(\pi^{\delta}|\gamma)$ as the sum of
the contributions on each time interval $[0,\delta]$,
$[\delta,2\delta]$ and $[2\delta,T]$. The first contribution vanishes
because $\pi^{\delta}$ solves the hydrodynamic equation in this
interval. On the time interval $[\delta,2\delta]$,
$\partial_t\rho^\delta_t = -\partial_t\lambda_{2\delta-t}=
-\Delta\varphi(\lambda_{2\delta-t}) =-\Delta\varphi(\rho^\delta_t)$.
In particular, the second contribution is equal to
\begin{eqnarray*}
\sup_{G\in \mc C^{1,2}_0(\overline{\Omega_T})}
\Big\{2\int_{0}^{\delta}ds\int_{\Omega}du\;
\nabla\varphi(\lambda)\cdot\nabla G -
\frac{1}{2}\int_{0}^{\delta}ds\; 
\langle\sigma(\lambda_t),\Vert\nabla G_t\Vert^2\rangle\Big\} 
\end{eqnarray*}
which, by Schwarz inequality, is bounded above by
\begin{equation*}
\int_0^{\delta}ds\int_{\Omega}du\;\varphi'(\lambda)
\frac{\Vert\nabla\lambda\Vert^2}{\chi(\lambda)}\, .
\end{equation*}

By Corollary \ref{rfhe}, this last expression converges to zero as
$\delta\downarrow 0$. Finally, the third contribution is bounded by
$I_T(\pi|\gamma)$ because $\pi^{\delta}$ in this interval is just a
time translation of the path $\pi$.
\end{proof}

Let $\Pi_2$ be the set of all paths $\pi$ in $\Pi_1$ with the property
that for every $\delta>0$ there exists $\epsilon>0$ such that
$\epsilon\leq\pi_t(\cdot) \leq 1-\epsilon$ for all $t\in[\delta,T]$.

\begin{lemma}
The set $\Pi_2$ is $I_T(\cdot|\gamma)$-dense.
\end{lemma}

\begin{proof}
By the previous lemma, it is enough to show that each path
$\pi(t,du)=\rho(t,u) du$ in $\Pi_1$ can be approximated by paths in
$\Pi_2$.  Fix $\pi$ in $\Pi_1$ and let $\lambda$ be as in the proof of
the previous lemma. For each $0<\varepsilon<1$, let
$\rho^{\varepsilon}=(1-\varepsilon)\rho+\varepsilon\lambda$,
$\pi^\varepsilon (t,du) = \rho^{\varepsilon}(t,u) du$.  Note that $\mc
Q(\pi^{\varepsilon})<\infty$ because $\mc Q$ is convex and both $\mc
Q(\pi)$ and $\mc Q(\lambda)$ are finite.  Hence, $\pi^{\varepsilon}$
belongs to $D_{\gamma}$ since both $\rho$ and $\lambda$ satisfy the
boundary conditions. Moreover, It is clear that $\pi^{\varepsilon}$
converges to $\pi$ as $\varepsilon\downarrow 0$.  By the lower
semicontinuity of $I_T(\cdot|\gamma)$, in order to conclude the proof,
it is enough to show that
\begin{eqnarray}\label{ls2}
\limsup_{N\to\infty}I_T(\pi^{\varepsilon}|\gamma)\leq I_T(\pi|\gamma)\, .
\end{eqnarray}

By Lemma \ref{lem05}, there exists $H\in H^1_0(\sigma(\rho))$ such that
$\rho$ solves the equation \eqref{f05}. Let ${\bf P} =
\sigma(\rho)\nabla H - \nabla\varphi(\rho)$ and ${\bf
  P}^{\lambda}=-\nabla\varphi(\lambda)$. For each $0<\varepsilon<1$,
let ${\bf P}^{\varepsilon} = (1-\varepsilon){\bf P} + \varepsilon {\bf
  P}^{\lambda}$.  Since $\rho$ solves the equation \eqref{f05}, for
every $G\in\mc C^{1,2}_0(\overline{\Omega_T})$,
\begin{equation*}
\int_0^T dt\;\langle{\bf P}^{\varepsilon}_t,\nabla G_t\rangle
=\langle\pi^{\varepsilon}_T,G_T\rangle-\langle\pi^{\varepsilon}_0,G_0\rangle
- \int_0^T dt\; \langle \pi^{\varepsilon}_t,\partial_tG_t\rangle \, . 
\end{equation*}
Hence, by \eqref{rf1}, $I_T(\pi^{\varepsilon}|\gamma)$ is equal to
\begin{equation*}
\sup_{G\in\mc C^{1,2}_0(\overline{\Omega_T})}
\Big\{\int_0^T dt\int_{\Omega}  \big\{
{\bf P}^{\varepsilon} + \nabla \varphi(\rho^{\varepsilon})
\big\} \cdot\nabla G \, du  \;-\; 
\frac{1}{2}\int_0^T dt\int_{\Omega} \sigma(\rho^{\varepsilon})
\Vert\nabla G\Vert^2 \, du \Big\}\;.
\end{equation*}
This expression can be rewritten as
\begin{equation*}
\begin{split}
& \frac{1}{2}\int_0^T dt\int_{\Omega} du\;\frac{\Vert{\bf
    P}^{\varepsilon}+\nabla \varphi(\rho^{\varepsilon}) \Vert^2}
{\sigma(\rho^{\varepsilon})} \\
& \qquad -\frac{1}{2}\inf_G \Big\{\int_0^T dt \int_{\Omega} 
\frac{\Vert{\bf P}^{\varepsilon}+\nabla \varphi(\rho^{\varepsilon})
-\sigma(\rho^{\varepsilon})\nabla G\Vert^2}{\sigma(\rho^{\varepsilon})}
\, du \Big\}   
\end{split}
\end{equation*}
Hence,
\begin{equation*}
I_T(\pi^{\varepsilon}|\gamma) \;\le\; \frac{1}{2}\, 
\int_0^T dt\int_{\Omega}  \frac{\Vert{\bf P}^{\varepsilon}
+\nabla \varphi(\rho^{\varepsilon}) \Vert^2}
{\sigma(\rho^{\varepsilon})}\, du \;\cdot 
\end{equation*}

In view of this inequality and \eqref{f06}, in order to prove
\eqref{ls2}, it is enough to show that
\begin{eqnarray*}
\lim_{\varepsilon\to 0} \int_0^T dt \int_{\Omega} du\;
\frac{\Vert{\bf P}^{\varepsilon}+
\nabla\varphi(\rho^{\varepsilon})\Vert^2}{\sigma(\rho^{\varepsilon})} \,du
\;=\; \int_0^T dt\int_{\Omega} 
\frac{\Vert{\bf P}+\nabla\varphi(\rho)\Vert^2}{\sigma(\rho)}\, du\;\cdot 
\end{eqnarray*}
By the continuity of $\varphi'$, $\sigma$ and from the definition of
${\bf P}^{\varepsilon}$, 
\begin{eqnarray*}
\lim_{\varepsilon \to 0} \frac{\Vert{\bf P}^{\varepsilon} 
+ \nabla\varphi (\rho^{\varepsilon})\Vert^2}
{\sigma(\rho^{\varepsilon})} \;=\;  
\frac{\Vert{\bf P} + \nabla\varphi(\rho)\Vert^2} {\sigma(\rho)} 
\end{eqnarray*}
almost everywhere. Therefore, to prove \eqref{ls2}, it remains to show
the uniform integrability of 
\begin{equation*}
\Big\{\frac{\Vert{\bf P}^{\varepsilon}\Vert^2}
{\chi(\rho^{\varepsilon})} \,:\, \varepsilon>0 \Big\}
\quad\hbox{ and }\quad 
\Big\{\frac{\Vert\nabla\rho^{\varepsilon}\Vert^2}
{\chi(\rho^{\varepsilon})}\,:\, \varepsilon>0 \Big\}\; . 
\end{equation*}

Since $I_T(\pi|\gamma)<\infty$, by \eqref{est2}, \eqref{f06} and
Corollary \ref{rfhe}, the functions $\frac{\Vert{\bf
    P}\Vert^2}{\chi(\rho)}$, $\frac{\Vert{\bf
    P}_{\lambda}\Vert^2}{\chi(\lambda)}$,
$\frac{\Vert\nabla\rho\Vert^2}{\chi(\rho)}$ and
$\frac{\Vert\nabla\lambda\Vert^2}{\chi(\lambda)}$ belong to
$L^1(\Omega_T)$. In particular, the function
\begin{equation*}
g=\max\left\{\frac{\Vert{\bf P}\Vert^2}
{\chi(\rho)},\frac{\Vert{\bf P}_{\lambda}\Vert^2}
{\chi(\lambda)},\frac{\Vert\nabla\rho\Vert^2}{\chi(\rho)},
\frac{\Vert\nabla\lambda\Vert^2}{\chi(\lambda)}\right\}\, ,
\end{equation*}
also belongs to $L^1(\Omega_T)$. By the convexity of
$\Vert\cdot\Vert^2$ an the concavity of $\chi(\cdot)$,
\begin{equation*}
\frac{\Vert{\bf P}^{\varepsilon}\Vert^2}{\chi(\rho^{\varepsilon})} 
\leq \frac{(1-\varepsilon)\Vert{\bf P}\Vert^2 +
\varepsilon\Vert{\bf P}_{\lambda}\Vert^2}
{(1-\varepsilon)\chi(\rho)+\varepsilon\chi(\lambda)}\leq g\, ,
\end{equation*}
which proves the uniform integrability of the family $\frac{\Vert{\bf
    P}^{\varepsilon}\Vert^2}{\chi(\rho^{\varepsilon})}$. The uniform
integrability of the family
$\frac{\Vert\nabla\rho_{\varepsilon}\Vert^2}{\chi(\rho_{\varepsilon})}$
follows from the same estimate with $\nabla\rho_{\varepsilon}$,
$\nabla\rho$ and $\nabla\lambda$ in the place of ${\bf
  P}_{\varepsilon}$, ${\bf P}$ and ${\bf P}_{\lambda}$, respectively.
\end{proof}

Let $\Pi$ be the subset of $\Pi_2$ consisting of all those paths $\pi$
which are solutions of the equation \eqref{f05} for some $H\in\mc
C^{1,2}_0(\overline{\Omega_T})$.

\begin{theorem}
\label{th5}
The set $\Pi$ is $I_T(\cdot|\gamma)$-dense.
\end{theorem}

\begin{proof}
By the previous lemma, it is enough to show that each path $\pi$ in
$\Pi_2$ can be approximated by paths in $\Pi$. Fix $\pi(t,du)
=\rho(t,u) du$ in $\Pi_2$.  By Lemma \ref{lem05}, there exists $H\in
H^1_{0}(\sigma(\rho))$ such that $\rho$ solves the equation
\eqref{f05}. Since $\pi$ belongs to $\Pi_2\subset \Pi_1$, $\rho$ is the
weak solution of \eqref{f02} in some time interval $[0,2\delta]$ for
some $\delta>0$. In particular, $\nabla H = 0$ a.e in
$[0,2\delta]\times\Omega$. On the other hand, since $\pi$ belongs to
$\Pi_1$, there exists $\epsilon>0$ such that $\epsilon \le
\pi_t(\cdot) \le 1-\epsilon$ for $\delta \le t\le T$.  Therefore,
\begin{equation}
\label{fin}
\int_0^T dt\int_{\Omega} \Vert\nabla H_t(u)\Vert^2 \, du
\;<\; \infty\, .
\end{equation}

Since $H$ belongs to $H^1_{0}(\sigma(\rho))$, there exists a sequence
of functions $\{H^n:\, n\geq 1\}$ in $\mc
C^{1,2}_{0}(\overline{\Omega_T})$ converging to $H$ in
$H^1_0(\sigma(\rho))$. We may assume of course that $\nabla H^n_t
\equiv 0$ in the time interval $[0,\delta]$. In particular,
\begin{eqnarray}
\label{lim3}
\lim_{n\to\infty}\int_0^T dt\int_{\Omega} du\; 
\Vert \nabla H^n_t(u)-\nabla H_t(u)\Vert^2 = 0\, .
\end{eqnarray}

For each integer $n>0$, let $\rho^n$ be the weak solution of
\eqref{f05} with $H^n$ in place of $H$ and set $\pi^n (t,du) =
\rho^n(t,u)du$. By \eqref{f06} and since $\sigma$ is bounded above in
$[0,1]$ by a finite constant,
\begin{equation*}
I_T(\pi^n|\gamma) = \frac{1}{2}\int_0^T\ dt\;
\langle\sigma(\rho^n_t),\Vert\nabla H^n_t\Vert^2\rangle
\leq C_0 \int_0^T dt\int_{\Omega} du\;
\Vert\nabla H^n_t(u)\Vert^2\, .
\end{equation*}
In particular, by \eqref{fin} and \eqref{lim3}, $I_T(\pi^n|\gamma)$ is
uniformly bounded on $n$. Thus, by Theorem \ref{th4}, the sequence
$\pi^n$ is relatively compact in $D([0,T],\mc M)$.

Let $\{\pi^{n_k}:\, k\geq 1\}$ be a subsequence of $\pi^n$ converging
to some $\pi^0$ in $D([0,T],\mc M^0)$. For every $G$ in $\mc
C^{1,2}_0(\overline{\Omega_T})$,
\begin{eqnarray*}
\langle\pi^{n_k}_T,G_T\rangle - \langle\gamma,G_0\rangle - 
\int_0^T dt\;\langle\pi^{n_k}_t,\partial_tG_t\rangle
=  \int_0^T dt\;\langle\varphi(\rho^{n_k}_t),\Delta G_t\rangle \\
- \int_0^T dt \int_{\Gamma}\varphi(b){\bf n_1}(\partial_{u_1}G)dS
- \int_0^T dt\; \langle\sigma(\rho^n_t),
\nabla H^{n_k}_t\cdot\nabla G_t\rangle\, .
\end{eqnarray*}

Letting $k\to\infty$ in this equation, we obtain the same equation
with $\pi^0$ and $H$ in place of $\pi^{n_k}$ and $H^{n_k}$,
respectively, if
\begin{equation}
\label{lim2}
\begin{split}
& \lim_{k\to\infty} \int_0^T dt\;\langle\varphi(\rho^{n_k}_t),\Delta
G_t\rangle \;=\; \int_0^T dt\;\langle\varphi(\rho^{0}_t),\Delta
G_t\rangle\;, \\
&\quad 
\lim_{k\to\infty}\int_0^T dt\; \langle\sigma(\rho^{n_k}_t),
\nabla H^{n_k}_t\cdot\nabla G_t\rangle = 
\int_0^T dt\;\langle\sigma(\rho^0_t),\nabla H_t\cdot\nabla G_t\rangle\, .
\end{split}
\end{equation}

We prove the second claim, the first one being simpler.  Note first
that we can replace $H^{n_k}$ by $H$ in the previous limit, because
$\sigma$ is bounded in $[0,1]$ by some positive constant and
\eqref{lim3} holds. Now, $\rho^{n_k}$ converges to $\rho^0$ weakly in
$L^2(\Omega_T)$ because $\pi^{n_k}$ converges to $\pi^0$ in
$D([0,T],\mc M^0)$. Since $I_T(\pi^n|\gamma)$ is uniformly bounded, by
Corollary \ref{corls} and Lemma \ref{lem02}, $\rho^{n_k}$ converges to
$\rho^0$ strongly in $L^2(\Omega_T)$ which implies \eqref{lim2}. In
particular, since \eqref{fin} holds, by uniqueness of weak solutions
of equation \eqref{f05}, $\pi^0 = \pi$ and we are done.
\end{proof}

\section{Large deviations}

We prove in this section the dynamical large deviations principle for
the empirical measure of boundary driven symmetric exclusion processes
in dimension $d\geq 1$. The proof relies on the results presented in
the previous section and is quite similar to the original one
presented in \cite{kov, dv}. There are just three additional
difficulties.  On the one hand, the lack of explicitly known
stationary states hinders the derivation of the usual estimates of the
entropy and the Dirichlet form, so important in the proof of the
hydrodynamic behaviour.  On the other hand, due to the definition of
the rate function, we have to show that trajectories with infinite
energy can be neglected in the large deviations regime. Finally, since
we are working with the empirical measure, instead of the empirical
density, we need to show that trajectories which are not absolutely
continuous with respect to the Lebesgue measure and whose density is
not bounded by one can also be neglected.  The first two problems have
already been faced and solved. The first one in \cite{LOV, BDGJL} and
the second in \cite{q, blm1}.  The approach here is quite similar, we
thus only sketch the main steps in sake of completeness.

\subsection{Superexponential estimates}

It is well known that one of the main steps in the derivation of the
upper bound is a super-exponential estimate which allows the
replacement of local functions by functionals of the empirical density
in the large deviations regime. Essentially, the problem consists in
bounding expressions such as $\langle V,f^2\rangle_{\mu^N_{ss}}$ in
terms of the Dirichlet form $\langle-N^2\mc L_N
f,f\rangle_{\mu^N_{ss}}$. Here $V$ is a local function and
$\langle\cdot,\cdot\rangle_{\mu^N_{ss}}$ indicates the inner product
with respect to the invariant state $\mu^N_{ss}$. In our context, the
fact that the invariant state is not known explicitly introduces a
technical difficulty.

Let $\beta$ be as in the beginning of section 2. Following \cite{LOV},
\cite{BDGJL}, we use $\nu^N_{\beta(\cdot)}$ as reference measure and
estimate everything with respect to $\nu^N_{\beta(\cdot)}$. However,
since $\nu^N_{\beta(\cdot)}$ is not the invariant state, there are no
reasons for $\langle-N^2\mc L_N f,f\rangle_{\nu^N_{\beta(\cdot)}}$ to
be positive. The next statement shows that this expression is almost
positive.

For each function $f:X_N\to\bb R$, let
\begin{equation*}
D_{N,0}(f) = \sum_{i=1}^d\sum_{x}\int r_{x,x+e_i}(\eta)
\left[f(\eta^{x,x+e_i})-f(\eta)\right]^2 d\nu^N_{\beta(\cdot)}(\eta)\, ,
\end{equation*}
where the second sum is carried over all $x$ such that $x,x+e_i\in
\Omega_N$.

\begin{lemma}
There exists a finite constant $C$ depending only on $\beta$ such that
\begin{equation*}
\langle N^2\mc L_{N,0}f, f\rangle_{\nu^N_{\beta(\cdot)}}
\leq -\frac{N^2}{4}D_{N,0}(f) + CN^d\langle f, 
f\rangle_{\nu^N_{\beta(\cdot)}}\, ,
\end{equation*}
for every function $f:X_N\to\bb R$.
\end{lemma}

The proof of this lemma is elementary and is thus omitted. Further, we
may choose $\beta$ for which there exists a constant $\theta>0$ such
that:
\begin{eqnarray*}
\beta(u_1,\check{u}) =  b(-1,\check{u}) & 
\qquad\hbox{ if } \; -1\leq u_1\leq -1+\theta\, , \\ 
\beta(u_1,\check{u}) =  b(1,\check{u})\;\;\, & 
\,\hbox{ if } \;1-\theta\leq u_1\leq 1\, ,
\end{eqnarray*}
for all $\check{u}\in\bb T^{d-1}$. In that case, for every $N$ large
enough, $\nu_{\beta(\cdot)}^N$ is reversible for the process with
generator $\mc L_{N,b}$ and then $\langle -N^2\mc L_{N,b}f,
f\rangle_{\nu^N_{\beta(\cdot)}}$ is positive.

This lemma together with the computation presented in \cite{BKL}, p.
78, for nonreversible processes, permits to prove the
super-exponential estimate.  For a cylinder function $\Psi$ denote the
expectation of $\Psi$ with respect to the Bernoulli product measure
$\nu^N_{\alpha}$ by $\widetilde{\Psi}(\alpha)$:
\begin{equation*}
\widetilde{\Psi}(\alpha) = E^{\nu^N_{\alpha}}[\Psi]\, .
\end{equation*}
For a positive integer $l$ and $x\in\Omega_N$, denote the empirical
mean density on a box of size $2l+1$ centered at $x$ by $\eta^l(x)$:
\begin{equation*}
\eta^l(x) = \frac{1}{|\Lambda_l(x)|}\sum_{y\in\Lambda_l(x)}\eta(y)\, ,
\end{equation*}
where
\begin{equation*}
\Lambda_l(x) = \Lambda_{N,l}(x) = \{y\in\Omega_N:\, |y-x|\leq l\}\, .
\end{equation*}
For each $G\in\mc C(\overline{\Omega_T})$, each cylinder function
$\Psi$ and each $\varepsilon>0$, let
\begin{equation*}
V_{N,\varepsilon}^{G,\Psi}(s,\eta)=\frac{1}{N^d}\sum_{x}
G(s,x/N)\left[\tau_x\Psi(\eta)-
\widetilde{\Psi}(\eta^{\varepsilon N}(x))\right]\, ,
\end{equation*}
where the sum is carried over all $x$ such that the support of
$\tau_x\Psi$ belongs to $\Omega_N$.

For a continuous function $H:[0,T]\times\Gamma\to\bb R$, let
\begin{equation*}
V_{N,H}^{\pm} = \int_0^T ds\;\frac{1}{N^{d-1}}
\sum_{x\in\Gamma_N^{\pm}}V^{\pm}(x,\eta_s)H
\left(s,\frac{x\pm e_1}{N}\right)\, ,
\end{equation*}
where $\Gamma_N^-$, resp. $\Gamma_N^+$, stands for the left, resp.
right, boundary of $\Omega_N$:
\begin{equation*}
\Gamma^{\pm}_N = \{(x_1,\cdots,x_d)\in \Gamma_N: x_1 = \pm(N-1)\}
\end{equation*}
and where
\begin{equation*}
V^{\pm}(x,\eta) = \left[\eta(x)+b\left(\frac{x\pm
      e_1}{N}\right)\right]
\left[\eta(x\mp e_1) - b\left(\frac{x\pm e_1}{N}\right)\right]\, .
\end{equation*}

\begin{proposition}
\label{see}
Fix $G\in\mc C(\overline{\Omega_T})$, $H$ in $\mc
C([0,T]\times\Gamma)$, a cylinder function $\Psi$ and a sequence
$\{\eta^N: N\geq 1\}$ of configurations with $\eta^N$ in $X_N$. For
every $\delta>0$,
\begin{eqnarray*}
&&\limsup_{\varepsilon\to 0}\limsup_{N\to\infty}
\frac{1}{N^d}\, \log \bb{P}_{\eta^N}
\Big[ \, \Big|\int_0^T V_{N,\varepsilon}^{G,\Psi}(s,\eta_s) \, ds \Big|
>\delta\Big] \;=\; -\infty\, , \\
&&\limsup_{N\to\infty}\frac{1}{N^d} \, \bb P_{\eta^N}
\big[ \, |V_{N,H}^{\pm}|>\delta \, \big] \;=\; -\infty\, .
\end{eqnarray*}
\end{proposition}

For each $\varepsilon> 0$ and $\pi$ in $\mc M$, denote by
$\Xi_\varepsilon (\pi) = \pi^{\varepsilon}$ the absolutely continuous
measure obtained by smoothing the measure $\pi$:
\begin{equation*}
\Xi_\varepsilon (\pi) (dx) \;=\; \pi^{\varepsilon} (dx) \;=\; 
\frac 1{U_\varepsilon} \frac {\pi(\bs \Lambda_\varepsilon(x))}
{|\bs \Lambda_\varepsilon(x)|} \,\, dx\;,
\end{equation*}
where $\bs \Lambda_\varepsilon(x) = \{y\in\Omega : |y-x|\le
\varepsilon\}$, $|A|$ stands for the Lebesgue measure of the set $A$,
and $\{U_\varepsilon : \varepsilon >0\}$ is a strictly decreasing
sequence converging to $1$: $U_\varepsilon >1$, $U_\varepsilon >
U_{\varepsilon'}$ for $\varepsilon>\varepsilon'$, $\lim_{\varepsilon
  \downarrow 0} U_\varepsilon = 1$. Let 
\begin{equation*}
\pi^{N,\varepsilon} \;=\; \Xi_\varepsilon (\pi^N)\;.
\end{equation*}
A simple computation shows that $\pi^{N,\varepsilon}$ belongs to $\mc
M^0$ for $N$ sufficiently large because $U_\varepsilon >1$, and that
for each continuous function $H:\Omega \to \bb R$,
\begin{equation*}
\<\pi^{N,\varepsilon}, H\> \;=\; \frac 1{N^d} \sum_{x\in
  \Omega_N} H(x/N) \eta^{\varepsilon N}(x) \; +\; O(N, \varepsilon)\;,
\end{equation*}
where $O(N, \varepsilon)$ is absolutely bounded by $C_0 \{ N^{-1} +
\varepsilon\}$ for some finite constant $C_0$ depending only on $H$.

For each $H$ in $\mc C_0^{1,2}(\overline{\Omega_T})$ consider the
exponential martingale $M_t^H$ defined by
\begin{eqnarray*}
M_t^H &=& \exp\Big\{N^d \Big[\big\langle\pi_t^N,H_t\big\rangle
-\big\langle\pi_0^N,H_0\big\rangle \\ 
&& \qquad\qquad\; -\; \frac{1}{N^d} \int_0^t
e^{-N^d\langle\pi_s^N,H_s\rangle} \, \big(\partial_s + 
N^2\mc L_N\big) \, e^{N^d\langle\pi_s^N,H_s\rangle} \, ds \Big]\Big\}\, .
\end{eqnarray*}
Recall from subsection 2.2 the definition of the functional $\hat
J_H$. An elementary computation shows that
\begin{eqnarray}
\label{mart}
M_T^H = \exp\left\{N^d\left[\hat J_{H}(\pi^{N, \varepsilon})
+\bb V_{N,\varepsilon}^H + c^1_H(\varepsilon) 
+ c^2_H(N^{-1})\right]\right\}\, .
\end{eqnarray}
In this formula, 
\begin{equation*}
\begin{split}
\bb V_{N,\varepsilon}^H \; &=\; -\sum_{i=1}^d\int_0^T 
V_{N,\varepsilon}^{\partial_{u_i}^2 H,h_i}(s,\eta_s) \, ds
- \frac{1}{2}\sum_{i=1}^d \int_0^T 
V_{N,\varepsilon}^{(\partial_{u_i}H)^2,g_i}(s,\eta_s) \, ds \\
\; & +\; a \,V^+_{N,\partial_{u_1}H}
\; -\; a\, V^-_{N,\partial_{u_1}H} 
\;+\; \langle\pi^N_0,H_0\rangle - \langle\gamma,H_0\rangle\, ;
\end{split}
\end{equation*}
the cylinder functions $h_i$, $g_i$ are given by
\begin{equation*}
\begin{split}
& h_i(\eta) \;=\; \eta(0) \;+\;  a \Big\{ \eta(0)[\eta(-e_i)+\eta(e_i)] 
- \eta(-e_i)\eta(e_i) \Big\}\, , \\
&\quad g_i(\eta) \;=\; r_{0,e_i}(\eta) \, [\eta(e_i)-\eta(0)]^2\; ;
\end{split}
\end{equation*}
and $c^j_H:\bb R_+ \to\bb R$, $j=1,2$, are functions depending only on
$H$ such that $c^j_H(\delta)$ converges to $0$ as $\delta\downarrow
0$.  In particular, the martingale $M_T^H$ is bounded by
$\exp\big\{C(H,T)N^d\big\}$ for some finite constant $C(H,T)$
depending only on $H$ and $T$. Therefore, Proposition \ref{see} holds
for $\bb P_{\eta^N}^H = \bb P_{\eta^N}M_T^H$ in place of $\bb
P_{\eta^N}$.

\subsection{Energy estimates}

To exclude paths with infinite energy in the large deviations regime,
we need an energy estimate. We state first the following technical
result.
\begin{lemma}
\label{est6}
There exists a finite constant $C_0$, depending on $T$, such that for
every $G$ in $C^{\infty}_c(\Omega_T)$, every integer $1\leq i\leq d$
and every sequence $\{\eta^N: N\geq 1\}$ of configurations with
$\eta^N$ in $X_N$,
\begin{eqnarray*}
\limsup_{N\to\infty}\frac{1}{N^d}\log\bb E_{\eta^N}
\Big[\exp\Big\{N^d\int_0^Tdt\; \langle\pi^N_t,
\partial_{u_i}G\rangle\Big\}\Big] \; \leq \; 
C_0\Big\{1+\int_0^T \Vert G_t\Vert^2_2 \, dt \Big\}\; .
\end{eqnarray*}
\end{lemma}

The proof of this proposition is similar to the one of Lemma A.1.1 in
\cite{KLO}.

Fix throughout the rest of the subsection a constant $C_0$ satisfying
the statement of Lemma \ref{est6}. For each $G$ in $\mc
C^{\infty}_c(\Omega_T)$ and each integer $1\leq i\leq d$, let
$\tilde{\mc Q}_i^G: D([0,T],\mc M)\to \bb R$ be the function given by
\begin{equation*}
\tilde{\mc Q}_i^G(\pi) =
\int_0^Tdt\;\langle\pi_t,\partial_{u_i}G_t\rangle
-C_0\int_0^Tdt\int_{\Omega} du\;G(t,u)^2 \;.
\end{equation*}

Notice that
\begin{eqnarray}
\label{f07}
\sup_{G\in\mc C^{\infty}_c(\Omega_T)}\left\{\tilde{\mc
    Q}_i^G(\pi)\right\} = \frac{\mc Q_i(\pi)}{4C_0}\, . 
\end{eqnarray}

Fix a sequence $\{G_k: k\geq 1\}$ of smooth functions dense in
$L^2([0,T], H^1(\Omega))$. For any positive integers $r,l$, let
\begin{equation*}
B_{r,l} = \Big\{\pi\in D([0,T],\mc M): \,\max_{\substack{1\leq k\leq
      r\\ 1\leq i\leq d}} \tilde{\mc Q}^{G_k}_i(\pi) \leq l\Big\}\, . 
\end{equation*}
Since, for fixed $G$ in $\mc C^{\infty}_c(\Omega_T)$ and $1\leq i\leq
d$ integer, the function $\tilde{\mc Q}_i^G$ is continuous, $B_{r,l}$
is a closed subset of $D([0,T],\mc M)$.

\begin{lemma}
\label{est5}
There exists a finite constant $C_0$, depending on $T$,  such that
for any positive integers $r,l$ and any sequence $\{\eta^N: N\geq 1\}$
of configurations with $\eta^N$ in $X_N$,
\begin{equation*}
\limsup_{N\to\infty} \frac{1}{N^d} 
\log Q_{\eta^N}\left[(B_{r,l})^c\right] \leq -l+C_0\, .
\end{equation*}
\end{lemma}

\begin{proof}
For integers $1\leq k\leq r$ and $1\leq i\leq d$, by Chebychev
inequality and by Lemma \ref{est6},
\begin{equation*}
\limsup_{N\to\infty} \frac{1}{N^d} 
\log \bb P_{\eta^N}\left[\tilde{\mc Q}_i^{G_k} > l\right] \leq -l +C_0\,.
\end{equation*}
Hence, from
\begin{eqnarray}
\label{ls}
\limsup_{N\to\infty} \frac{1}{N^d}\log(a_N+b_N) 
\leq \max\left\{\limsup_{N\to\infty} 
\frac{1}{N^d}\log a_N,\limsup_{N\to\infty} \frac{1}{N^d}\log b_N\right\}\, ,
\end{eqnarray}
we obtain the desired inequality.
\end{proof}

\subsection{Upper Bound}

Fix a sequence $\{F_k: k\geq 1\}$ of smooth nonnegative functions
dense in $\mc C(\overline{\Omega})$ for the uniform topology.  For
$k\ge 1$ and $\delta>0$, let
\begin{equation*}
D_{k,\delta} = \Big\{\pi\in D([0,T],\mc M): \, 0\le \<\pi_t , F_k\>
\le \int_{\Omega} F_k(x) \, dx  \,+\, C_k \delta 
\;,\, 0\le t\le T \Big\} \, ,
\end{equation*}
where $C_k = \Vert \nabla F_k \Vert_\infty$ and $\nabla F$ is the
gradient of $F$. Clearly, the set $D_{k,\delta}$, $k\ge 1$, $\delta
>0$, is a closed subset of $D([0,T],\mc M)$. Moreover, if
\begin{equation*}
E_{m,\delta} \;=\; \bigcap_{k=1}^m D_{k,\delta}\;,
\end{equation*}
we have that $D([0,T],\mc M^0) = \cap_{n\ge 1} \cap_{m\ge 1}
E_{m,1/n}$. Note, finally, that for all $m\ge 1$, $\delta>0$,
\begin{equation}
\label{ff01}
\pi^{N,\varepsilon} \text{ belongs to $E_{m,\delta}$ for $N$ sufficiently
large.}
\end{equation}
\smallskip

Fix a sequence of configurations $\{\eta^N: N\geq 1\}$ with $\eta^N$
in $X_N$ and such that $\pi^N(\eta^N)$ converges to $\gamma(u)du$ in
$\mc M$. Let $A$ be a subset of $D([0,T],\mc M)$,
\begin{equation*}
\frac{1}{N^d}\log\bb P_{\eta^N}\left[\pi^N\in A\right] 
= \frac{1}{N^d}\log \bb E_{\eta^N}\left[M_T^H \, (M_T^H)^{-1}
\, {\bf 1} \{\pi^N\in A\}\right]\, .
\end{equation*}
Maximizing over $\pi^N$ in $A$, we get from \eqref{mart} that the last
term is bounded above by
\begin{equation*}
-\inf_{\pi\in A} \hat J_H(\pi^{\varepsilon})
+\frac{1}{N^d}\log\bb E_{\eta^N}
\Big[M_T^H \, e^{-N^d\bb V_{N,\varepsilon}^H}\Big] 
- c^1_H(\varepsilon)-c^2_H(N^{-1})\, .
\end{equation*}
Since $\pi^N(\eta^N)$ converges to $\gamma(u)du$ in $\mc M$ and since
Proposition \ref{see} holds for $\bb P_{\eta^N}^H = \bb
P_{\eta^N}M_T^H$ in place of $\bb P_{\eta^N}$, the second term of the
previous expression is bounded above by some $C_H(\varepsilon, N)$
such that
\begin{equation*}
\limsup_{\varepsilon\to 0}\limsup_{N\to \infty}
C_H(\varepsilon, N) = 0\, .
\end{equation*}
Hence, for every $\varepsilon>0$, and every $H$ in $\mc
C^{1,2}_0(\overline{\Omega_T})$,
\begin{eqnarray}
\label{est7}
\limsup_{N\to\infty}\frac{1}{N^d}\log\bb P_{\eta^N}[A]
\leq -\inf_{\pi\in A} \hat J_H(\pi^{\varepsilon})+C'_H(\varepsilon)\, ,
\end{eqnarray}
where $\displaystyle \lim_{\varepsilon\to 0}C'_H(\varepsilon)=0$.

For each $H\in \mc C_0^{1,2}(\overline{\Omega_T})$, each
$\varepsilon>0$ and any $r,l, m, n \in\bb Z_+$, let
$J_{H,\varepsilon}^{r,l,m,n}:D([0,T],\mc M)\to\bb R\cup\{\infty\}$ be
the functional given by
\begin{equation*}
J_{H,\varepsilon}^{r,l,m,n}(\pi) = 
\begin{cases}
\hat J_H(\pi^{\varepsilon}) & \hbox{ if } \pi\in B_{r,l} \cap
E_{m,1/n} \, ,\\
+\infty & \hbox{ otherwise } .
\end{cases}
\end{equation*}
This functional is lower semicontinuous because so is $\hat J_H \circ
\Xi_\varepsilon$ and because $B_{r,l}$, $E_{m,1/n}$ are closed subsets
of $D([0,T],\mc M)$.

Let $\mc O$ be an open subset of $D([0,T],\mc M)$. By Lemma
\ref{est5}, \eqref{ls}, \eqref{ff01} and \eqref{est7},
\begin{eqnarray*}
\limsup_{N\to\infty}\frac{1}{N^d}\log Q_{\eta^N}[\mc O] 
& \leq & \max\Big\{\limsup_{N\to\infty}\frac{1}{N^d}
\log Q_{\eta^N}[\mc O\cap B_{r,l} \cap E_{m,1/n}] \, , \\ 
&& \qquad\qquad\qquad\qquad
\;\limsup_{N\to\infty}\frac{1}{N^d}\log Q_{\eta^N}[(B_{r,l})^c]\Big\}
\\ & \leq & \max\Big\{-\inf_{\pi\in \mc O\cap B_{r,l} \cap E_{m,1/n}}
\hat J_H(\pi^{\varepsilon}) + C'_H(\varepsilon)\, , \,-l+C_0\Big\}
\\ & = & - \inf_{\pi\in\mc O} L_{H,\varepsilon}^{r,l,m,n}(\pi) \, ,
\end{eqnarray*}
where
\begin{equation*}
L_{H,\varepsilon}^{r,l,m,n}(\pi)=\min \left\{J_{H,\varepsilon}^{r,l,m,n}(\pi)
- C'_H(\varepsilon)\, , \,l- C_0\right\}\, .
\end{equation*}
In particular,
\begin{equation*}
\limsup_{N\to\infty}\frac{1}{N^d}\log Q_{\eta^N}[\mc O]
\leq - \sup_{H,\varepsilon,r,l,m,n}\;\inf_{\pi\in\mc O}
L_{H,\varepsilon}^{r,l,m,n}(\pi)\, .
\end{equation*}

Note that, for each $H\in \mc C_0^{1,2}(\overline{\Omega_T})$, each
$\varepsilon>0$ and $r,l,m,n \in\bb Z_+$, the functional
$L_{H,\varepsilon}^{r,l,m,n}$ is lower semicontinuous. Then, by Lemma
A2.3.3 in \cite{KL}, for each compact subset $\mc K$ of $D([0,T],\mc
M)$,
\begin{eqnarray*}
\limsup_{N\to\infty}\frac{1}{N^d}\log Q_{\eta^N}[\mc K] 
\;\leq\; - \inf_{\pi\in\mc K}\;\sup_{H,\varepsilon,r,l,m,n}
L_{H,\varepsilon}^{r,l,m,n}(\pi) \, .
\end{eqnarray*}
By \eqref{f07} and since $D([0,T],\mc M^0) = \cap_{n\ge 1}
\cap_{m\ge 1} E_{m,1/n}$,
\begin{equation*}
\begin{split}
& \limsup_{\varepsilon\to 0}\limsup_{l\to\infty}
\limsup_{r\to\infty}\limsup_{m\to\infty}\limsup_{n\to\infty} 
L_{H,\varepsilon}^{r,l,m,n}(\pi) \; = \\
&\qquad\qquad\qquad\qquad\qquad 
\begin{cases}
\hat J_H(\pi) & \hbox{ if } \mc Q(\pi)<\infty \text{ and } \pi \in
D([0,T],\mc M^0)\, ,\\
+\infty & \hbox{ otherwise }.
\end{cases}
\end{split}
\end{equation*}
This result and the last inequality imply the upper bound for compact
sets because $\hat J_H$ and $J_H$ coincide on $D([0,T],\mc M^0)$.  To
pass from compact sets to closed sets, we have to obtain exponential
tightness for the sequence $\{Q_{\eta^N}\}$. This means that there
exists a sequence of compact sets $\{\mc K_n : \,n\geq 1\}$ in
$D([0,T],\cm)$ such that
\begin{equation*}
\limsup_{N\to\infty}\frac{1}{N^d}\log Q_{\eta^N}({\mc K_n}^c)\leq -n\, .
\end{equation*}
The proof presented in \cite{B} for the non interacting zero range
process is easily adapted to our context.

\subsection{Lower Bound}

The proof of the lower bound is similar to the one in the convex
periodic case. We just sketch it and refer to \cite{KL}, section 10.5.
Fix a path $\pi$ in $\Pi$ and let $H\in\mc
C^{1,2}_0(\overline{\Omega_T})$ be such that $\pi$ is the weak
solution of equation \eqref{f05}. Recall from the previous section the
definition of the martingale $M_t^H$ and denote by $\bb P^H_{\eta^N}$
the probability measure on $D([0,T],X_N)$ given by $\bb
P_{\eta^N}^H[A] = \bb E_{\eta^N}[M_T^H \mb 1 \{A\}]$. Under $\bb
P^H_{\eta^N}$ and for each $0\leq t\leq T$, the empirical measure
$\pi^N_t$ converges in probability to $\pi_t$. Further,
\begin{equation*}
\lim_{N\to\infty}\frac{1}{N^d}H
\left(\bb P^H_{\eta^N}\big|\bb P_{\eta^N}\right) = I_T(\pi|\gamma)\, ,
\end{equation*}
where $H(\mu|\nu)$ stands for the relative entropy of $\mu$ with
respect to $\nu$. From these two results we can obtain that for every
open set $\mc O\subset D([0,T],\mc M)$ which contains $\pi$,
\begin{equation*}
\liminf_{N\to\infty}\frac{1}{N^d}\log
\bb P_{\eta^N}\big[\mc O\big]\geq -I_T(\pi|\gamma)\, .
\end{equation*}
The lower bound follows from this and the $I_T(\cdot|\gamma)$-density
of $\Pi$ established in Theorem \ref{th5}.

\section{Existence and uniqueness of weak solutions}
\label{sec5}

We prove in this section existence and uniqueness of weak solutions of
the boundary value problems \eqref{f01} and \eqref{f02}, as well as
some properties of the solutions. We start with the parabolic
differential equation.

\begin{proposition}
\label{s05}
Let $\rho_0 :\overline{\Omega} \to [0,1]$ be a measurable function.
There exists a unique weak solution of \eqref{f02}.
\end{proposition}

\begin{proof}
Existence of weak solutions of \eqref{f02} is warranted by the
tightness of the sequence $\Qss^N$ proved in Section \ref{sec3}.
Indeed, fix a profile $\rho_0:\Omega\to [0,1]$ and consider a sequence
$\{\mu^N : N\ge 1\}$ of probability measures in $\mc M$ associated to
$\rho_0$ in the sense \eqref{f03}. Fix $T>0$ and denote by $Q^N$ the
probability measure on $D([0,T], \mc M)$ induced by the measure
$\mu^N$ and the process $\pi^N_t$. Repeating the arguments of Section
\ref{sec3}, one can prove that the sequence $\{Q^N:N\ge 1\}$ is tight
and that any limit point of $\{Q^N:N\ge 1\}$ is concentrated on weak
solutions of \eqref{f03}. This proves existence. Uniqueness follows
from Lemma \ref{lem1-ann} below.
\end{proof}

Denote by $\Vert \cdot \Vert_1$ the $L^1(\Omega)$ norm. Next lemma
states that the $L^1(\Omega)$-norm of the difference of two weak
solutions of the boundary value problem \eqref{f02} decreases in time: 

\begin{lemma}
\label{lem1-ann}
Fix two profiles $\rho_0^1$, $\rho_0^2: \Omega \to [0,1]$.  Let
$\rho^j$, $j=1$, $2$, be weak solutions of \eqref{f02} with initial
condition $\rho_0^j$. Then, $\|\rho_t^1 -\rho_t^2\|_1$ decreases in
time. In particular, there is at most one weak solution of
\eqref{f02}.
\end{lemma}

\begin{proof}
Fix two profiles $\rho_0^1$, $\rho_0^2: \Omega \to [0,1]$.  Let
$\rho^j$, $j=1$, $2$, be weak solutions of \eqref{f02} with initial
condition $\rho_0^j$. Fix $0\le s< t$.  For $\delta >0$ small, denote
by $R_\delta$ the function defined by
\begin{equation*} 
R_\delta (u) \; =\; \frac{u^2}{2\delta }\mb 1\{|u|\le \delta \}
\; +\; \big( | u| -\delta/2\big)
 \mb 1\{|u| >\delta\} \; .
\end{equation*}

Let $\psi: \bb R^d \to \bb R_+$ be a smooth approximation of the
identity:
\begin{equation*}
  \psi(u) \ge 0\;,\quad \text{supp } \psi \subset [-1,1]^d\;, \quad
\int \psi(u)\, du = 1\;.
\end{equation*}
For each positive $\epsilon$, define $\psi_\epsilon$ as
\begin{equation*}
  \psi_\epsilon (u) = \epsilon^{-d} \psi(u\epsilon^{-1})\;.
\end{equation*}

Taking the time derivative of the convolution of $\rho_t^j$ with
$\psi_\epsilon$, after some elementary computations based on
properties ({\bf H1}), ({\bf H2}) of weak solutions of \eqref{f02},
one can show that
\begin{eqnarray*}
&& \int_\Omega du \, R_\delta \big (\rho^1(t,u) -\rho^2(t,u) \big) 
-\int_\Omega du \, R_\delta \big (\rho^1(s,u) -\rho^2(s,u) \big)
\\ 
&&\quad = \; - \delta^{-1} \int_s^t d\tau \int_{A_\delta} du\,
\nabla (\rho^1 -\rho^2) \cdot \big\{  \varphi'(\rho^1) \nabla
\rho^1 -\varphi' (\rho^2) \nabla \rho^2\big\}\;,
\end{eqnarray*}
where $A_\delta$ stands for the subset of $[0,T]\times \Omega$ where
$|\rho^1(t,u) - \rho^2(t,u)| \le \delta$. We may rewrite the previous
expression as
\begin{equation*}
\aligned 
&\; - \delta^{-1} \int_s^t d\tau  \int_{A_\delta} du\,
\varphi' (\rho^1) \|  \nabla (\rho^1 -\rho^2) \|^2 \\ 
&\quad -\; \delta^{-1}\int_s^t d\tau \int_{A_\delta} du\,
\big\{ \varphi'(\rho^1)- \varphi'(\rho^2) \big\}
\nabla (\rho^1 -\rho^2) \cdot 
\nabla \rho^2 \; . 
\endaligned 
\end{equation*} 

Since $\rho^1,\rho^2$ are positive and bounded by $1$, there exists a
positive constant $c_0$ such that $c_0\le \varphi'
(\rho^j(\tau,u))$. The first line in the previous formula is then
bounded above by
\begin{equation*}
- c_0 \delta^{-1} \int_s^t d\tau  \int_{A_\delta} du\, \Vert \nabla
(\rho^1 -\rho^2) \Vert^2\; .
\end{equation*}
On the other hand, since $\varphi'$ is Lipschitz, on the set
$A_\delta$, $| \varphi'(\rho^1)- \varphi' (\rho^2)| \le M| \rho^1
-\rho^2| \le M\delta$ for some positive constant $M$.  In particular,
by Schwarz inequality, the second line of the previous formula is
bounded by
\begin{equation*} 
\delta^{-1} M A \int_s^t d\tau  \int_{A_\delta} du\, \Vert \nabla
(\rho^1 -\rho^2) \Vert^2\; +\; 
\delta M A^{-1} \int_s^t d\tau  \int_{A_\delta} du\, 
\Vert \nabla \rho^2 \Vert^2
\end{equation*} 
for every $A>0$. Choose $A=M^{-1} c_0$ to obtain that
\begin{equation*}
\aligned 
&\int_\Omega du \, R_\delta \big(\rho^1(t,u) -\rho^2(t,u) \big)
-\int_\Omega du \, R_\delta \big (\rho^1(s,u) -\rho^2(s,u) \big)\\
&\qquad \le \; \delta c_0^{-1} M^2  \int_0^t d\tau  \int du\, 
\Vert \nabla \rho^2 \Vert^2 \; .
\endaligned 
\end{equation*} 
Letting $\delta\downarrow 0$, we conclude the proof of the lemma
because $R_\delta (\cdot)$ converges to the absolute value function as
$\delta\downarrow 0$. 
\end{proof}

\begin{lemma}
\label{uniqH}
Fix two profiles $\rho_0^1$, $\rho_0^2: \Omega \to [0,1]$.  Let
$\rho^j$, $j=1$, $2$, be weak solutions of \eqref{f05} for the same
$H$ satisfying \eqref{fin} and with initial condition $\rho_0^j$.
Then, $\|\rho_t^1 -\rho_t^2\|_1$ decreases in time. In particular,
there is at most one weak solution of \eqref{f05} when $H$ satisfies
\eqref{fin}.
\end{lemma}

\begin{proof}
Following the same procedure of the proof of the previous lemma, we
get first
\begin{eqnarray*}
&& \int_\Omega du \, R_\delta \big (\rho^1(t,u) -\rho^2(t,u) \big) 
-\int_\Omega du \, R_\delta \big (\rho^1(s,u) -\rho^2(s,u) \big)
\\ 
&&\quad = \; - \delta^{-1} \int_s^t d\tau \int_{A_\delta} du\,
\nabla (\rho^1 -\rho^2) \cdot \big\{  \varphi'(\rho^1) \nabla
\rho^1 -\varphi' (\rho^2) \nabla \rho^2\big\}
\\
&&\qquad\,\,- \delta^{-1} \int_s^t d\tau \int_{A_\delta} du\,
\big\{\sigma(\rho^1)-\sigma(\rho^2)\big\}\nabla (\rho^1 -\rho^2) \cdot \nabla H
\;,
\end{eqnarray*}
and then
\begin{eqnarray*}
&&\int_\Omega du \, R_\delta \big(\rho^1(t,u) -\rho^2(t,u) \big)
-\int_\Omega du \, R_\delta \big (\rho^1(s,u) -\rho^2(s,u) \big)
\\
&&\quad \leq \; \delta C_1  \int_0^t d\tau  \int du\, 
\Vert \nabla \rho^2 \Vert^2  + \delta C_2\int_0^t d\tau  \int du\, 
\Vert \nabla H \Vert^2\; ,
\end{eqnarray*}
for some positive constants $C_1$ and $C_2$. Hence, letting
$\delta\downarrow 0$ we conclude the proof of the lemma.
\end{proof}

The same ideas permit to show the monotonicity of weak solutions of
\eqref{f02}. This is the content of the next result which plays a
fundamental role in proving existence and uniqueness of weak solutions
of \eqref{f01}.

\begin{lemma}
\label{lembis-ann} 
Fix two profiles $\rho_0^1$, $\rho_0^2: \Omega \to [0,1]$.  Let
$\rho^j$, $j=1$, $2$, be the weak solutions of (\ref{f02}) with
initial condition $\rho_0^j$. Assume that there exists $s\ge 0$ such
that
\begin{equation*}
\lambda \big\{ u\in \Omega\ :\ \ 
\rho^1(s,u) \le\rho^2(s,u)  \big\}=1\; , 
\end{equation*} 
where $\lambda$ is the Lebesgue measure on $\Omega$. Then, for all
$t\ge s$
\begin{equation*}
\lambda \big\{ u\in \Omega\ :\ \ 
\rho^1(t,u) \le\rho^2(t,u)  \big\}=1\;.
\end{equation*}
\end{lemma}

\begin{proof}
We just need to repeat the same proof of the Lemma \ref{lem1-ann} 
by considering the function $R_\delta^+(u) =R_\delta(u) \mb 1\{u \ge
0\}$ instead of $R_\delta$.
\end{proof}

\begin{corollary}
\label{s03}
Denote by $\rho^0$ (resp. $\rho^1$) the weak solution of \eqref{f02}
associated to the initial profile constant equal to $0$
(resp. $1$). Then, for $0\le s\le t$, $\rho^1_t (\cdot) \le \rho^1_s
(\cdot)$ and $\rho^0_s(\cdot) \le \rho^0_t(\cdot)$ a.e.
\end{corollary}

\begin{proof}
Fix $s\ge 0$. Note that $\hat \rho(r,u)$ defined by $\hat \rho(r,u) =
\rho^1(s+r,u)$ is a weak solution of \eqref{f02} with initial
condition $\rho^1(s,u)$. Since $\rho^1(s,u) \le 1 = \rho^1(0,u)$, by
the previous lemma, for all $r\ge 0$, $\rho^1(r+s,u) \le \rho^1(r,u)$
for almost all $u$.
\end{proof}

We now turn to existence and uniqueness of the boundary value problem
\eqref{f01}. Recall the notation introduced in the beginning of
Section \ref{sc}.  Consider the following classical
boundary-eigenvalue problem for the Laplacian:
\begin{equation}
\label{eq:laplace-eq}
 \left\{ \begin{array}{lll}
 -\Delta U =\alpha  U \,,  & \  \\
  U \in  H^1_0(\Omega) \, . & \ \end{array}
     \right. 
\end{equation}
By the Sturm--Liouville theorem (cf. \cite{E}, Subsection 9.12.3),
problem (\ref{eq:laplace-eq}) has a countable system $\{U_n,\alpha_n :
n\ge 1\}$ of eigensolutions which contains all possible eigenvalues.
The set $\{U_n : n\ge 1\}$ of eigenfunctions forms a complete
orthonormal system in the Hilbert space $L^2(\Omega)$, each $U_n$
belong to $H^1_0 (\Omega)$, all the eigenvalues $\alpha_n$, have
finite multiplicity and
\begin{equation*}
0<\alpha_1\le \alpha_2\le \cdots \le \alpha_n \le \cdots \to \infty \; .
\end{equation*}
The set $\{U_n/\alpha_n^{1/2} : n\ge 1\}$ is a complete orthonormal
system in the Hilbert space $H^1_0 (\Omega)$.  Hence, a function $V$
belongs to $L^2(\Omega)$ if and only if
\begin{equation*}
V \;=\; \lim_{n\to \infty} \sum_{k=1}^n \<V,U_k\>_2 \, U_k 
\end{equation*}
in $L^2(\Omega)$. In this case,
\begin{equation*}
\<V,W\>_2 \;=\;  \sum_{k=1}^\infty  \<V,U_k\>_2\,
\overline{ \<W,U_k\>_2 }
\end{equation*}
for all $W$ in $L^2(\Omega)$. Moreover, a function $V$ belongs to
$H^1_0(\Omega)$ if and only if
\begin{equation*}
V \;=\; \lim_{n\to \infty} \sum_{k=1}^n \<V,U_k\>_2 \, U_k 
\end{equation*}
in $H^1_0(\Omega)$. In this case,
\begin{equation}
\label{f04}
\<V,W\>_{1,2,0} \;=\;  \sum_{k=1}^\infty  \alpha_k \<V,U_k\>_2\,
\overline{ \<W,U_k\>_2 }
\end{equation}
for all $W$ in $H^1_0(\Omega)$. 

\begin{lemma}
\label{s02}
Fix two profiles $\rho_0^1$, $\rho_0^2: \Omega \to [0,1]$.  Let
$\rho^j$, $j=1$, $2$, be the weak solutions of (\ref{f02}) with
initial condition $\rho_0^j$. Then, 
\begin{equation*}
\int_{0}^\infty  \|\rho_t^1 -\rho_t^2 \|_1^2  \, dt
\;<\; \infty\; .
\end{equation*}
In particular,
\begin{equation*}
\lim_{t\to \infty} \|\rho_t^1 -\rho_t^2\|_1 \; =0 \; \; .
\end{equation*}
\end{lemma}

\begin{proof}
Fix two profiles $\rho_0^1$, $\rho_0^2: \Omega \to [0,1]$ and let
$\rho^j$, $j=1$, $2$, be the weak solutions of (\ref{f02}) with initial
condition $\rho_0^j$. Let $\rho^j_t (\cdot)= \rho^j(t,\cdot)$.  For
$n\ge 1$ let $F_n:\bb R_+ \to \R$ be the function defined by
\begin{equation*}
F_n(t) =\sum_{k=1}^n \frac{1}{\alpha_k} \big| \< \rho_t^1-\rho_t^2
\, , \, U_k\>_2\big|^2\; .
\end{equation*}

Since $\rho^1,\rho^2 $ are weak solutions, $F_n$ is time
differentiable. Since $\Delta U_k =-\alpha_k U_k$ and since
$\alpha_k>0$, for $t>0$,
\begin{equation}
\begin{aligned}
\frac{d}{dt} F_n(t) = &- \sum_{k=1}^n 
\Big\{ \< \rho_t^1-\rho_t^2 \,,\, U_k \>_2 \,
\overline{ \< \varphi(\rho_t^1) -\varphi(\rho_t^2)\,,\, U_k \>_2 }\\
\ &\qquad\qquad\qquad\qquad
+ \< \varphi(\rho_t^1) -\varphi(\rho_t^2) \,,\,  U_k \>_2 \,
\overline{ \< \rho_t^1-\rho_t^2 \,,\, U_k \>_2 } \Big\}\, .
\end{aligned}
\label{eq:laplace-eq3}
\end{equation}

Fix $t_0 >0$. Integrating (\ref{eq:laplace-eq3}) in time, applying
identity (\ref{f04}), and letting $n\uparrow \infty$, we get
\begin{equation*}
\begin{aligned}
\int_{t_0}^T dt \int_\Omega 
\big[ \varphi(\rho_t^1(u)) -\varphi(\rho_t^2(u))\big]
     \big[ \rho_t^1(u)-\rho_t^2(u)\big]  du \;=  &
\lim_{n \to \infty} \frac{1}{2} \Big\{ F_n(t_0) -F_n(T)\Big\}\\
\; \le & \; \frac{1}{2 \alpha_1}     
 \| \rho_{t_0}^1-\rho_{t_0}^2 \|_2^2 
\end{aligned}
\end{equation*}
for all $T>t_0$. Since $\rho_{t_0}^1-\rho_{t_0}^2$ belongs to $L^2
(\Omega)$,
\begin{equation*}
\int_{t_0}^\infty dt \int_\Omega 
\big[ \varphi(\rho_t^1(u)) -\varphi(\rho_t^2(u))\big]
\big[ \rho_t^1(u)-\rho_t^2(u)\big]\, du \;<\; \infty\; .
\end{equation*}

There exists a positive constant $C_2$ such that, for all $a,b \in
[0,1]$
\begin{equation*}
C_2 ( b- a)^2 \le \big(\varphi (b) -\varphi(a) \big) (b-a)\; .
\end{equation*}
On the other hand, by Schwarz inequality, for all $t\ge t_0$,
\begin{equation*}
\|\rho_t^1 -\rho_t^2 \|_1^2 \le 2 \|\rho_t^1 -\rho_t^2 \|_2^2\; . 
\end{equation*}
Therefore
\begin{equation*}
\int_{t_0}^\infty  \|\rho_t^1 -\rho_t^2 \|_1^2 \,  dt
<\infty\; .
\end{equation*}
and the first statement of the lemma is proved because the integral
between $[0,t_0]$ is bounded by $4t_0$.  The second statement of the
lemma follows from the first one and from Lemma \ref{lem1-ann}.
\end{proof}

\begin{proposition}
\label{s04}
There exists a unique weak solution of the boundary value problem
\eqref{f01}.
\end{proposition}

\begin{proof}
We start with existence. Let $\rho^1 (t,u)$ (resp. $\rho^0(t,u)$) be
the weak solution of the boundary value problem \eqref{f02} with
initial profile constant equal to $1$ (resp. $0$). By Lemma
\eqref{lembis-ann}, the sequence of profiles $\{\rho^1(n, \cdot) :
n\ge 1\}$ (resp. $\{\rho^0(n, \cdot) : n\ge 1\}$) decreases (resp.
increases) to a limit denoted by $\rho^+(\cdot)$ (resp.
$\rho^-(\cdot)$). In view of Lemma \ref{s02}, $\rho^+ = \rho^-$ almost
surely. Denote this profile by $\bar\rho$ and by $\bar\rho(t,\cdot)$
the solution of \eqref{f02} with initial condition $\bar\rho$. Since
$\rho^0(t, \cdot)\le \bar\rho (\cdot)\le \rho^1(t,\cdot)$ for all
$t\ge 0$, by Lemma \ref{lembis-ann}, $\rho^0(t+s, \cdot)\le \bar\rho
(s,\cdot)\le \rho^1(t+s,\cdot)$ a.e.  for all $s$, $t\ge 0$. Letting
$t\uparrow\infty$, we obtain that $\bar\rho (s,\cdot)=
\bar\rho(\cdot)$ a.e. for all $s$. In particular, $\bar\rho$ is a solution
of \eqref{f01}.

Uniqueness is simpler. Assume that $\rho^1$, $\rho^2:\Omega\to [0,1]$
are two weak solution of \eqref{f01}. Then, $\rho^j(t,u) = \rho^j(u)$,
$j=1$, $2$, are two stationary weak solutions of \eqref{f02}. By Lemma
\ref{s02}, $\rho^1 = \rho^2$ almost surely.
\end{proof}

\end{document}